\documentclass[reqno,a4paper, 11pt]{amsart}

\usepackage[a4paper=true,pdfpagelabels]{hyperref}
\usepackage{graphicx}

\usepackage[ansinew]{inputenc}
\usepackage{amsfonts,epsfig}
\usepackage{latexsym}
\usepackage{mathabx}
\usepackage{amsmath}
\usepackage{amssymb}

\usepackage{color}

\newtheorem{theorem}{Theorem}
\newtheorem{lemma}[theorem]{Lemma}
\newtheorem{corollary}[theorem]{Corollary}

\newtheorem{proposition}[theorem]{Proposition}
\newtheorem{question}[theorem]{Question}
\newtheorem{lettertheorem}{Theorem}
\newtheorem{letterlemma}[lettertheorem]{Lemma}

\theoremstyle{definition}

\theoremstyle{remark}

\numberwithin{equation}{section}

\newcommand{\abs}[1]{\lvert#1\rvert}
\newcommand{\nm}[1]{\lVert#1\rVert}

\newcommand{\B}{\mathcal{B}}
\newcommand{\D}{\mathbb{D}}
\newcommand{\DD}{\widehat{\mathcal{D}}}
\newcommand{\Dd}{\widecheck{\mathcal{D}}}

\newcommand{\DDD}{\mathcal{D}}

\newcommand{\N}{\mathbb{N}}

\newcommand{\RR}{\mathbb{R}}

\newcommand{\ZZ}{\mathcal{Z}}
\newcommand{\C}{\mathbb{C}}

\newcommand{\e}{\varepsilon}

\renewcommand{\phi}{\varphi}
\newcommand{\A}{\mathcal{A}}

\newcommand{\T}{\mathbb{T}}

\newcommand{\whw}{\widehat{\omega}}

\def\BMOA{\mathord{\rm BMOA}}

\def\HL{\mathord{\rm HL}}

           \def\e{\varepsilon}
     \def\om{\omega}      
       \def\t{\theta}       
                  \def\z{\zeta}
                  \def\vp{\varphi}
\def\G{\Gamma}

\DeclareMathOperator{\supp}{supp}
\renewcommand{\H}{\mathcal{H}}

\begin{document}
\title{Volterra-type operators mapping weighted Dirichlet space into $H^\infty$}




\author{Jos\'e \'Angel Pel\'aez}
\address{Departamento de An\'alisis Matem\'atico, Universidad de M\'alaga, Campus de
Teatinos, 29071 M\'alaga, Spain} \email{japelaez@uma.es}

\author{Jouni R\"atty\"a}
\address{University of Eastern Finland, P.O.Box 111, 80101 Joensuu, Finland}
\email{jouni.rattya@uef.fi}

\author{Fanglei Wu}
\address{University of Eastern Finland, P.O.Box 111, 80101 Joensuu, Finland}
\email{fanglei.wu@uef.fi; fangleiwu1992@gmail.com}

\thanks{The research was supported in part by Ministerio de Econom\'{\i}a y Competitividad, Spain, projects PGC2018-096166-B-100; La Junta de Andaluc{\'i}a, projects FQM210 S and UMA18-FEDERJA-002, and Vilho, Yrj\"o ja Kalle V\"ais\"al\"a foundation of Finnish Academy of Science and Letters.}

\begin{abstract}
The problem  of describing the analytic functions  $g$ on the unit disc such that the integral operator $T_g(f)(z)=\int_0^zf(\z)g'(\z)\,d\z$
is bounded (or compact) from a Banach space (or complete metric space) $X$ of analytic functions to the Hardy space $H^\infty$ is a tough problem, and remains unsettled in many cases. For analytic functions $g$ with non-negative Maclaurin coefficients, we describe the boundedness and compactness
of $T_g$ acting from a weighted Dirichlet space $D^p_\omega$, induced by an upper doubling weight $\omega$, to $H^\infty$.
We also characterize, in terms of neat conditions on $\omega$, the upper doubling weights for which $T_g: D^p_\omega\to H^\infty$ is bounded (or compact) only if $g$ is constant.
\end{abstract}

\maketitle

\section{Introduction and main results}

Let $\H(\D)$ denote the space of analytic functions in the unit disc $\D=\{z\in\C:|z|<1\}$. Each $g\in\H(\D)$ induces the integral operator, also called Volterra-type operator, defined by
	$$
	T_g(f)(z)=\int_0^zf(\zeta)g'(\zeta)\,d\zeta,\quad z\in\D.
	$$
The study of this operator has attracted a substantial amount of attention during the last decades since the publications of the seminal works~\cite{AC,AS,Pom}.
Characterizing the operators $T_g$, mapping a Banach space (or complete metric space) $X\subset\H(\D)$ boundedly or compactly to the Hardy space $H^\infty$, in terms of a reasonable condition depending on the symbol $g$ only is known to be difficult, and remains unsettled in many cases~\cite{AJS,CPPR,SSVAdv2017}. Consequently, restricting $g$ to some subclass of $\H(\D)$, sets a natural approach to the problem. In \cite{SSVAdv2017} the authors described the univalent symbols $g$ such that $T_g: H^\infty \to H^\infty$ is bounded. Further, the functions $g\in \H(\D)$ with non-negative Taylor coefficients such that $T_g(H^p)\subset H^\infty$ were described in \cite{PRW}. In this paper, we are mainly interested in the situation where $X$ is a weighted Dirichlet space and, in some of the results, the symbol $g$ has non-negative Maclaurin coefficients.

For a non-negative function $\om\in L^1([0,1))$, its extension to $\D$, defined by $\om(z)=\om(|z|)$ for all $z\in\D$, is called a radial weight. For $0<p<\infty$ and such an $\omega$, the Berman space $A^p_\om$ consists of $f\in\H(\D)$
such that
    $$
    \|f\|_{A^p_\omega}^p=\int_\D|f(z)|^p\omega(z)\,dA(z)<\infty,
    $$
where $dA(z)=\frac{dx\,dy}{\pi}$ is the normalized Lebesgue area measure on $\D$.
The corresponding weighted Dirichlet space is
	$$
	D^p_\omega=\left\{f\in\H(\D):\left\|f\right\|_{D^p_\omega}^p=\|f'\|_{A^p_\omega}^p+|f(0)|^p<\infty\right\}.
	$$
The classical weighted Dirichlet space $D^p_{\alpha}$ is, by definition, equal to $D^p_\omega$, induced by the standard radial weight $\om(z)=(\alpha+1)(1-|z|^2)^{\alpha}$, where $-1<\alpha<\infty$. Throughout this paper we assume $\widehat{\om}(z)=\int_{|z|}^1\om(s)\,ds>0$ for all $z\in\D$, for otherwise $D^p_\omega=\H(\D)=A^p_\om$.

A radial weight $\om$ belongs to the class~$\DD$ if there exists a constant $C=C(\om)\ge1$ such that the tail integral $\widehat{\om}$ satisfies the doubling condition $\widehat{\om}(r)\le C\widehat{\om}(\frac{1+r}{2})$ for all $0\le r<1$. Moreover, if there exist constants $K=K(\om)>1$ and $C=C(\om)>1$ such that $\widehat{\om}(r)\ge C\widehat{\om}\left(1-\frac{1-r}{K}\right)$ for all $0\le r<1$, then we write $\om\in\Dd$. Finally, the intersection $\DD\cap\Dd$ is denoted by $\DDD$.

It is worth noticing that, if there exists $\nu\in\DDD$ such that
	\begin{equation}\label{eq:intro1}
	\omega(z)\asymp\nu(z)(1-|z|)^{p},\quad z\in\D,
	\end{equation}
then $D^p_{\om}$ coincides with the weighted Bergman space $A^p_\nu$ by \cite[Theorem~5]{PelRat2020}. However,
$\widehat{\omega}(z)$ may decrease to zero arbitrarily slowly, as $z$ approaches to the boundary, and therefore \eqref{eq:intro1} may very well fail. Typical examples of weights violating \eqref{eq:intro1} are $\om(z)=(\alpha+1)(1-|z|^2)^{\alpha}$ with $-1<\alpha<p-1$. Moreover, weights in $\DD$ may have a wild oscillatory behavior and they may even vanish on sets that are not hyperbolically uniformly bounded. In these cases \eqref{eq:intro1} also certainly fails. Illuminating examples of weights in the deceivingly simply looking class~$\DD$ are given in~\cite[Proposition~10]{PR2020} and~\cite[Proposition~12]{PRBMO22}.

We begin with considering the question of when $T(D^p_{\om}, H^{\infty})$ consists of constant functions only, provided $\om\in\DD$. From now on, if $X\subset\H(\D)$ is a Banach space (or complete metric space), $T(X,H^\infty)$ (resp. $T_c(X,H^\infty)$) denotes the set of $g\in\H(\D)$ such that $T_g:X\to H^\infty$ is bounded (resp. compact). It is known that $T(H^p,H^\infty)$ consists of constant functions only, if only if, $0<p<1$, by \cite[Theorem~2.5(vi)]{CPPR}. Therefore each $T(A^p_\om,H^\infty)$ contains constant functions only, provided $0<p<1$ and $\omega$ is any radial weight. As expected, the situation is different in the case of the weighted Dirichlet space $D^p_\omega$. The first result of this paper characterizes the triviality of $T(D^p_{\om}, H^{\infty})$ and $T_c(D^p_{\om}, H^{\infty})$ in terms of a neat condition on $p$ and $\om$.

\begin{theorem}\label{intro:0<p<1}
Let $0<p\le1$, $\om\in\DD$ and $g\in H^{\infty}$. Then the following statements are equivalent:
\begin{itemize}
\item[(i)] $T(D^p_{\om}, H^{\infty})$ consists of constant functions only;
\item[(ii)] $\displaystyle \sup_{0\leq r<1}\frac{(1-r)^{2-\frac1p}}{\whw(r)^{\frac1p}}=\infty$;
\item[(iii)] $I:D^p_{\om}\rightarrow D^1_0$ is unbounded.
\end{itemize}
Moreover,
\begin{itemize}
\item[(i)] $T_c(D^p_{\om}, H^{\infty})$ consists of constant functions only;
\item[(ii)] $\displaystyle \sup_{0\leq r<1}\frac{(1-r)^{2-\frac1p}}{\whw(r)^{\frac1p}}>0$;
\item[(iii)] $I:D^p_{\om}\rightarrow D^1_0$ is not compact.
\end{itemize}
\end{theorem}

Probably the most important part of the theorem is the surprising equivalence between the behavior of the embedding $D^1_0\subset D^p_{\om}$ and that of the integration operator $T_g:D^p_{\om}\to H^{\infty}$. The true difficulty in the proof lies in showing that (i) implies (ii), because it is not that complicated to show that (iii) implies (i), and the equivalence between (ii) and (iii) is just the Carleson embedding theorem~\cite[Theorem~1]{PR2015}. Before further comments on the proof, we state the result concerning the range $1<p<\infty$. We underline here that when $1<p<\infty$ there is no such neat connection to embeddings as in the case $0<p\le1$, as is shown in Section~\ref{Sec:p>1}, see Theorem~\ref{Dpconstant} and the discussion presiding it.

\begin{theorem}\label{intro:1p<infinity}
Let $1<p<\infty$, $\om\in\DD$ and $g\in H^{\infty}$. Then $T(D^p_{\om}, H^{\infty})$ (equivalently $T_c(D^p_{\om}, H^{\infty})$) consists of constant functions only, if and only if,
	$$
	\int_0^1\frac{(1-r)^{p'}}{\widehat{\om}(r)^{p'-1}}\,dr=\infty.
	$$
\end{theorem}

Theorem~\ref{intro:1p<infinity} and its proof allow us to extend \cite[Theorem~1.3]{CPPR} from the setting of so-called regular weights to the whole doubling class $\DD$. Being precise, we deduce that $T(A^p_{\om}, H^{\infty})$ consists of constant functions only, if and only if,
	$$
	\int_0^1\frac{dr}{\widehat{\om}(r)^{p'-1}}=\infty,
	$$
provided $1<p<\infty$, $\om\in\DD$ and $g\in H^{\infty}$. Details yielding to this conclusion are given at the very end of the paper, in Section~\ref{Sec:p>1}.

Apart from the geometric characterizations of Carleson embeddings for the weighted Bergman spaces~\cite{PR2015}, the proofs of Theorems~\ref{intro:0<p<1} and \ref{intro:1p<infinity} are based on two main ingredients: an abstract criterion of the continuity of integral operators~\cite[Theorem~2.2]{CPPR}, and appropriate duality relations for the space $D^p_\omega$, given in Section~\ref{section:dualities}. In particular, we will show that, for $0<p<1$, the dual of $D^p_{\om}$ is isomorphic to the Zygmund space $\ZZ$ via the pairing
	\begin{equation*}\label{D^1-pairing}
	\langle f,g\rangle_{D^2_{W}}=\langle f',g'\rangle_{A^2_{W}}+f(0)\overline{g(0)},
	\end{equation*}
with equivalence of norms, see Lemma~\ref{lemma:dual-Dirichlet} in Section~\ref{section:dualities}. Here $W$ is an appropriate radial weight which depends on $p$ and $\omega$. Recall that the Zygmund space $\ZZ$ consists of $f\in\H(\D)$ such that
	$$
	\|f\|_{\ZZ}=\sup_{z\in\D}|f''(z)|(1-|z|^2)+|f'(0)|+|f(0)|<\infty.
	$$
It is well-known that $\ZZ$ is a subset of the disc algebra $\A$, and the containment of $f$ in $\ZZ$ is characterized by the boundary value condition
	$$
	\sup_{\theta}\sup_{h>0}\frac{|f(e^{i(\theta+h)}+f(e^{i(\theta-h)})-2f(e^{i\theta})|}{h}<\infty.
	$$
It is worth observing that $(D^p_{p-1})^\star$ is isomorphic to the weighted Bloch space $\B^2$
via the $A^2_{\frac1p-1}$-pairing, provided $0<p\le 1$~\cite[Lemma~6]{PRW}. The interplay between these different identifications of the dual of $D^p_{\om}$ is of course the change of the pairing with respect to which there are taken. The advantage of our duality $(D^p_{\om})^\star\simeq\ZZ$ compared to $(D^p_{\om})^\star\simeq\B^2$ is that, apart from being a much more general result, its proof is easier than that of \cite[Lemma~6]{PRW}. Namely, the crucial step in the proof of the last-mentioned result relies on technical tools related to coefficient multipliers of the Bloch spaces. Because of the pairing we work with, we can avoid many tedious calculations all together, and the proof itself becomes more straightforward and transparent via an appropriate use of a Carleson embedding theorem for the weighted Bergman spaces.

The dual of $D^1_\om$ is described via a suitable weighted $A^2$-pairing in terms of a weighted $\BMOA$-type space in Lemma~\ref{le:duald1omhat} in Section~\ref{section:dualities}. This result has its roots in the description of $(A^1_\om)^\star$ given in \cite[Theorem~4]{PelRat2020}. While at first glance the last-mentioned duality, as well as our description of $(D^1_\om)^\star$, might look intuitively unclear, this involved result serves us to resolve the case $p=1$ in Theorem~\ref{intro:0<p<1} and other forthcoming results.

The techniques developed on the way to the aforementioned results also allow us to characterize the weights $\omega\in\DD$ such that $T_c(D^p_{\om}, H^{\infty})$ contains constant functions only. The statements are given in Sections~\ref{Sec:p<1}--\ref{Sec:p>1} when each of the three cases $0<p<1$, $p=1$ and $1<p<\infty$ are considered separately in the said order.

The other set of results we obtain concern symbols $g$ with non-negative Maclaurin coefficients such that $T_g: D^p_\omega\to H^\infty$ is bounded or compact. These characterizations are provided in terms of the moments of the weights. Therefore we write $\omega_x=\int_0^1 r^x\omega(r)\,dr$ for all $x\ge 0$. From now on, set $g(z)=\sum_{n=0}^\infty\widehat{g}(n)z^n, z\in\D.$

\begin{theorem}\label{nonnegativepless1intro}
Let $\omega\in\DD$ and $g\in H^\infty$ such that $\widehat{g}(n)\geq0$ for all $n\in\mathbb{N}\cup\{0\}$. Then the following statements are valid:
\begin{enumerate}
\item[\rm(i)] If $0<p<1$, then $T_g:D^p_{\om}\to H^\infty$ is bounded if and only if
	$$
	A_{p,\omega}=\sup_{0\leq r<1}\left((1-r)\sum_{k=0}^{\infty}\frac{(k+1)^{\frac1p-1}r^k}{(\om_k)^{\frac1p}}\left(\sum_{n=0}^{\infty}
	\frac{\widehat{g}(n+1)(n+1)}{n+k+1}\right)\right)<\infty,
	$$
and
	$
	\|T_g\|_{D^p_{\om}\rightarrow H^\infty}
	\asymp A_{p,\omega}.$
Moreover, $T_g:D^p_{\om}\to H^\infty$ is compact if and only if
	$$
	\limsup_{r\rightarrow1^-}\left((1-r)\sum_{k=0}^{\infty}\frac{(k+1)^{\frac1p-1}r^k}{(\om_k)^{\frac1p}}\left(\sum_{n=0}^{\infty}
	\frac{\widehat{g}(n+1)(n+1)}{n+k+1}\right)\right)=0.
	$$
\item[\rm(ii)] If $p=1$, then $T_g:D^p_{\om}\to H^\infty$ is bounded if and only if
  $$
	A_{\omega}=\sup_{0<r,s<1}\left(\whw(r)^2\sum_{m=0}^\infty\frac{1-s}{(m+1)^2}\left(\sum_{k=0}^m\frac{r^{k+2}}
	{(\om\whw)_{k+1}}\sum_{n=0}^\infty\frac{\widehat{g}(n+1)(n+1)}{n+k+1}s^{n+m+3}\right)^2\right)<\infty,
	$$
and
	$
	\|T_g\|_{D^1_{\om}\rightarrow H^\infty}^2
	\asymp A_{\omega}.
	$
Moreover, $T_g:D^1_{\om}\to H^\infty$ is compact if and only if
$$
\limsup_{r\rightarrow1^-}\limsup_{s\rightarrow1^-}\left(\whw(r)^2\sum_{m=0}^\infty\frac{1-s}{(m+1)^2}\left(\sum_{k=0}^m\frac{r^{k+2}s^{m-k}}
{(\om\whw)_{k+1}}\sum_{n=0}^\infty\frac{\widehat{g}(n+1)(n+1)}{n+k+1}\right)^2\right)=0.
$$
\item[\rm(iii)] If $1<p<\infty$, then $T_g: D^p_{\omega}\rightarrow H^\infty$ is bounded (equivalently compact) if and only if
	\begin{equation*}
	\sum_{k=0}^{\infty}\frac{(k+1)^{-2}}{(\omega_k)^{p'-1}}
	\left(\sum_{n=0}^{\infty}\frac{(n+1)\widehat{g}(n+1)}{n+k+1}\right)^{p'}<\infty.
	\end{equation*}
Moreover,
	\begin{equation*}
	\|T_g\|_{D^p_{\om}\rightarrow H^{\infty}}^{p'}\asymp\sum_{k=0}^{\infty}\frac{(k+1)^{-2}}{(\omega_k)^{p'-1}}
	\left(\sum_{n=0}^{\infty}\frac{(n+1)\widehat{g}(n+1)}{n+k+1}\right)^{p'}.
	\end{equation*}
	\end{enumerate}
\end{theorem}

The conditions appearing in the three cases in Theorem~\ref{nonnegativepless1intro} are different due to the different identifications of the dual of $D^p_\omega$ used. This fact forces us to employ different techniques in each case of the theorem. In particular, we will use a decomposition norm theorem for $D^p_\omega$ \cite[Theorem~3.4]{PeR}, see also \cite[Theorem~4]{PR2017}, valid for $1<p<\infty$ and~$\omega\in\DD$, and results on universal C\'esaro basis of polynomials.

It is worth mentioning that the hypothesis $g\in H^\infty$ in Theorems~\ref{intro:0<p<1}-\ref{nonnegativepless1intro} is not a real restriction, because it is obviously necessary so that $T_g: D^p_\om\to H^\infty$ to be bounded.

There is one more thing worth mentioning before proceeding to the proofs. Namely, some of our arguments take us very naturally to consider spaces defined in terms of the Maclaurin coefficients of the function. To this end, for $0<p<\infty$ and a radial weight $\om$, define the weighted Hardy-Littlewood space $\HL^{\om}_{p}$ as the set of those $f\in\H(\D)$ whose Maclaurin coefficients $\{\widehat{f}(n)\}_{n=0}^\infty$ satisfy
	$$
	\nm{f}^p_{\HL_p^{\om}}=\sum_{n=0}^{\infty}\abs{\widehat{f}(n)}^p (n+1)^{p-2}\om_{np+1}<\infty.
	$$
In the next section we explain how this space come to the picture and what it serves us for.

We finish the introduction by couple of words about the notation used in this paper. Throughout the paper, $\frac1p+\frac{1}{p'}=1$ for $1<p<\infty.$ Further, the letter $C=C(\cdot)$ will denote an absolute constant whose value depends on the parameters indicated
in the parenthesis, and may change from one occurrence to another. If there exists a constant
$C=C(\cdot)>0$ such that $a\le Cb$, then we write either $a\lesssim b$ or $b\gtrsim a$. In particular, if $a\lesssim b$ and
$a\gtrsim b$, then we denote $a\asymp b$ and say that $a$ and $b$ are comparable.

\section{New spaces and basic results}

Recall that, for $0<p<\infty$ and $-1<\alpha<\infty$, the classical weighted Dirichlet space $D^p_\alpha$ is induced by the standard radial weight $\om(z)=(\alpha+1)(1-|z|^2)^{\alpha}$. The closely related Hardy-Littlewood space $\HL_p$ contains those $f\in\H(\D)$ whose Maclaurin coefficients $\{\widehat{f}(n)\}_{n=0}^\infty$ satisfy
	$$
	\nm{f}^p_{\HL_p}=\sum_{n=0}^{\infty}\abs{\widehat{f}(n)}^p (n+1)^{p-2}<\infty.
	$$
These and the Hardy spaces satisfy the well-known continuous embeddings
	\begin{equation}\label{Hardy1}
	D^p_{p-1}\subset H^p\subset\HL_p,\quad 0<p\le 2,
	\end{equation}
and
	\begin{equation}\label{Hardy2}
	\HL_p\subset H^p\subset D^p_{p-1} ,\quad 2\le p<\infty,
	\end{equation}
by \cite{D,Flett,LP}. Each of these inclusions is strict unless $p=2$ in which case all the spaces are the same by direct calculations or straightforward applications of Parseval's formula and Green's theorem.

Recall that, for $0<p<\infty$ and a radial weight $\om$, the weighted Hardy-Littlewood space $\HL^{\om}_{p}$ was defined by the condition
	$$
	\nm{f}^p_{\HL_p^{\om}}=\sum_{n=0}^{\infty}\abs{\widehat{f}(n)}^p (n+1)^{p-2}\om_{np+1}<\infty,
	$$
where $\om_x=\int_0^1r^x\om(r)\,dr$ for all $0\le x<\infty$. Since obviously $\om_x\to0$, as $x\to\infty$, we have $\HL_p\subset\HL_p^\om$ for each radial weight $\omega$. The spaces $\HL^{\om}_{p}$ arise naturally when \eqref{Hardy1} and \eqref{Hardy2} are applied to dilatations and integrated over $(0,1)$ with respect to $r\om(r)$. To see this in detail, we will need the following lemma which will be used repeatedly also throughout the rest paper, see \cite[Lemma 2.1]{PelSum14} for a proof.

\begin{letterlemma}\label{Eqdhat}
Let $\om$ be a radial weight. Then the following conditions are equivalent:
\begin{itemize}
\item[(i)] $\om\in\DD$;
\item[(ii)] There exist $C=C(\om)>0$ and $\beta_0=\beta_0(\om)>0$ such that
	$$
	\widehat{\om}(r)\leq C\left(\frac{1-r}{1-t}\right)^{\beta}\widehat{\om}(t),\quad 0\leq r\leq t<1,
	$$
for all $\beta\geq\beta_0$;
\item[(iii)] The asymptotic equality
	$$
	\om_x=\int_0^1s^x\om(s)\,ds\asymp\whw\left(1-\frac1x\right),
	$$
is valid for any $x\geq1$;
\item[(iv)] There exist constants $C=C(\om)\ge1$ and $\eta=\eta(\om)>0$ such that	
	$$
	\om_x\le C\left(\frac{y}{x}\right)^\eta\om_{y},\quad 0<x\le y<\infty.
	$$
\end{itemize}
\end{letterlemma}

Set $\widetilde{\om}(z)=\frac{\widehat{\om}(z)}{(1-|z|)}$ and $\om_{[\beta]}(z)=\om(z)(1-|z|)^\beta$ for all $\beta\in\RR$ and $z\in\D$. If $\om\in\DD$, then the spaces $\HL^{\om}_p$, $A^p_{\om}$ and $D^p_{\widetilde{\om}_{[p]}}$ are closely related and obey the inclusions corresponding to the norm inequalities appearing in the following result.

\begin{proposition}\label{inclusion}
Let $\om\in\DD$. Then the following statements hold:
\begin{itemize}
\item[(i)] If $0<p\le2$, then $\|f\|_{\HL_p^{\om}}\lesssim\|f\|_{A^p_{\om}}\lesssim\|f\|_{D^p_{\widetilde{\om}_{[p]}}}$ for all $f\in\H(\D)$;
\item[(ii)] If $2\le p<\infty$, then $\|f\|_{D^p_{\widetilde{\om}_{[p]}}}\lesssim\|f\|_{A^p_{\om}}\lesssim\|f\|_{\HL_p^{\om}}$ for all $f\in\H(\D)$.
\end{itemize}
\end{proposition}

\begin{proof}
We begin with (i). Denote $f_r(z)=f(rz)$ for all $0\le r<1$ and $z\in\D$. An application of \eqref{Hardy1} to $f_r$ and a change of variable yield
	\begin{align*}
	\sum_{n=0}^{\infty}|\widehat{f}(n)|^pr^{np}(n+1)^{2-p}
	&\lesssim M_p^p(r,f)
	\lesssim\int_0^1\left(\int_0^{2\pi}\abs{f'(rte^{i\theta})}^p\,d\theta\right)r^p(1-t)^{p-1}t\,dt+|f(0)|^p\\
	&=\int_{0}^{r}M_p^p(s,f')\left(r-s\right)^{p-1}\frac{s}{r}\,ds+|f(0)|^p,\quad f\in\H(\D).
	\end{align*}
By multiplying this inequality by $\om(r)r$, integrating from 0 to 1 with respect to $r$, and using Fubini's theorem we deduce
	\begin{equation}\label{Eq:1}
	\begin{split}
	\|f\|_{\HL^{\om}_p}^p
	&=\sum_{n=0}^{\infty}(n+1)^{2-p}\abs{\widehat{f}(n)}^p\om_{np+1}
	\lesssim \|f\|^p_{A^p_{\om}}\\
	&\lesssim\int_0^1\left(\int_0^rM_p^p(s,f')
	\left(r-s\right)^{p-1}s\,ds\right)\om(r)\,dr+|f(0)|^p\\
	&=\int_0^1M_p^p(s,f')\left(\int_s^1\left(r-s\right)^{p-1}\om(r)\,dr\right)s\,ds+|f(0)|^p
	=\|f\|^p_{D^p_{h_{\om}}},
	\end{split}
	\end{equation}
where
	$$
	h_{\om}(z)=\int_{|z|}^1\left(r-|z|\right)^{p-1}\om(r)\,dr,\quad z\in\D.
	$$
	
If $1\le p<\infty$, then the pointwise inequality $h_{\om}(z)\le\widehat{\om}(z)(1-|z|)^{p-1}=\widetilde{\om}_{[p]}(z)$ is valid for all $z\in\D$. This observation together with \eqref{Eq:1} proves (i) with $1\le p\le2$ without any hypothesis on the radial weight $\om$.

If $0<p<1$, then $h_{\om}\ge\widetilde{\om}_{[p]}$ on $\D$, and hence we must argue in a different manner. By the hypothesis $\om\in\DD$ we deduce
	\begin{equation*}
	\int_r^1\widetilde{\om}_{[p]}(t)\,dt
	\lesssim\int_r^1\widehat{\om}\left(\frac{1+t}{2}\right)(1-t)^{p-1}\,dt
	=2^p\int_{\frac{1+r}{2}}^1\widetilde{\om}_{[p]}(t)\,dt,\quad 0\le r<1,
	\end{equation*}
and thus $\widetilde{\om}_{[p]}\in\DD$ for each $0<p<\infty$. Therefore the estimate $\|f\|_{D^p_{h_{\om}}}\lesssim\|f\|_{D^p_{\widetilde{\om}_{[p]}}}$ follows from \cite[Theorem~1(b)]{PR2015} at once if we show that $h_{\om}(S(a))\lesssim\widetilde{\om}_{[p]}(S(a))$ for all $a\in\D$. Since $h_{\om}$ and $\widetilde{\om}_{[p]}$ both are radial, we may ignore the angular integral and prove this as follows. By Fubini's theorem and the hypothesis $\om\in\DD$ we deduce
	\begin{equation}\label{Eq:CM-condition}
	\begin{split}
	\int_{|a|}^1h_{\om}(s)\,sds
	&\le\int_{|a|}^1\om(r)\left(\int_{|a|}^r\left(r-s\right)^{p-1}\,ds\right)\,dr
	=\frac1p\int_{|a|}^1\left(r-|a|\right)^p\om(r)\,dr\\
	&\le\frac1p\widehat{\om}(a)(1-|a|)^p
	\lesssim\widehat{\om}\left(\frac{1+|a|}{2}\right)\int_{|a|}^{\frac{1+|a|}{2}}(1-t)^{p-1}\,dt\\
	&\le\int_{|a|}^1\widetilde{\om}_{[p]}(t)\,dt,\quad a\in\D.
	\end{split}
	\end{equation}
It follows that, for each $0<p<1$ and $\om\in\DD$, we have $\|f\|_{D^p_{h_{\om}}}\lesssim\|f\|_{D^p_{\widetilde{\om}_{[p]}}}$ for all $f\in\H(\D)$. This finishes the proof of (i).


To prove (ii) we first observe that, by arguing as in the first part of the proof, an application of \eqref{Hardy2} yields $\|f\|^p_{D^p_{h_{\om}}}\lesssim\|f\|^p_{A^p_{\om}}\lesssim\|f\|_{\HL^{\om}_p}^p$ for all $f\in\H(\D)$. Further, if $1<p<\infty$, then an integration by parts and the hypothesis $\om\in\DD$ yield
	\begin{align*}
	h_{\om}(z)
	&=(p-1)\int_{|z|}^1\left(r-|z|\right)^{p-2}\widehat{\om}(r)\,dr
	\ge(p-1)\widehat{\om}\left(\frac{3+|z|}{4}\right)\int_{\frac{1+|z|}{2}}^{\frac{3+|z|}{4}}\left(r-|z|\right)^{p-2}\,dr\\
	&=\widehat{\om}\left(\frac{3+|z|}{4}\right)(1-|z|)^{p-1}\left(\left(\frac{3}{4}\right)^{p-1}-\left(\frac{1}{2}\right)^{p-1}\right)
	\gtrsim\widetilde{\om}_{[p]}(z),\quad z\in\D.
	\end{align*}
This pointwise estimate yields $\|f\|_{D^p_{\widetilde{\om}_{[p]}}}\lesssim\|f\|_{D^p_{h_{\om}}}$ for all $f\in\H(\D)$, provided $1<p<\infty$ and $\om\in\DD$. Therefore (ii) is proved.

To this end we observe that the proof above actually shows that, for each $0<p<\infty$ and $\om\in\DD$, we have $\|f\|_{D^p_{h_{\om}}}\asymp\|f\|_{D^p_{\widetilde{\om}_{[p]}}}$ for all $f\in\H(\D)$. Therefore there is no loss of information when passing from $D^p_{h_{\om}}$, which arises naturally when \eqref{Hardy1} and \eqref{Hardy2} are applied to dilatations and integrated, to the space $D^p_{\widetilde{\om}_{[p]}}$ if $\om\in\DD$.
\end{proof}

Our next goal is to show that the inclusions derived from Proposition~\ref{inclusion} might be strict unless $p=2$.

\begin{proposition}\label{H-A-different}
Let $\om$ be a radial weight and $0<p<\infty$. If $p\neq2$, then $\HL^{\om}_{p}$ and $A^p_{\om}$ are two different spaces.
\end{proposition}

\begin{proof}
If the spaces were the same, then there would exist a constant $C=C(p,\omega)>1$ such that
$C^{-1}\|f\|_{\HL^{\om}_{p}}\le\|f\|_{A^p_\omega }\le C\|f\|_{\HL^{\om}_{p}}$ for all $f\in\H(\D)$. But this is impossible because the monomials $m_n(z)=z^n$ satisfy
	$$
	\frac{\|m_n\|_{A^p_\om}^p}{2}=\om_{np+1}=\frac{\|m_n\|_{\HL^p_\om}^p}{(n+1)^{2-p}},\quad n\in\N,
	$$
and this yields a contradiction as $n\to\infty$.
\end{proof}

The next proposition concerns the spaces $D^p_{\widetilde{\om}_{[p]}}$ and $A^p_{\om}$. It can be established by following the argument used in the proof of \cite[Proposition~4.3]{PR}.

\begin{proposition}
There exists $\om\in\DD\setminus\DDD$ such that $D^p_{\widetilde{\om}_{[p]}}$ and $A^p_{\om}$ are two different spaces, provided $p\neq2$.
\end{proposition}

The last auxiliary result stated in this section concerns analytic functions with non-negative Maclaurin coefficients tending to zero.

\begin{lemma}
Let $1\le p<\infty$ and $\om\in\DD$. Then
	$$
	\|f\|_{\HL_p^{\om}}\asymp\|f\|_{A^p_{\om}}\asymp\|f\|_{D^p_{\widetilde{\om}_{[p]}}}
	$$
for all $f\in\H(\D)$ such that its Maclaurin coefficients $\{\widehat{f}(n)\}_{n=0}^\infty$ form a sequence of non-negative numbers decreasing to zero.
\end{lemma}

\begin{proof}
It is well known that, for each $1\le p<\infty$, we have
	\begin{equation}\label{Eq:old-case}
	\|f\|^p_{H^p}\asymp\|f\|^p_{D^p_{p-1}}\asymp\|f\|^p_{\HL_p}
	\end{equation}
for all functions $f$ as in the statement on the lemma, see
\cite{HD}, \cite{P2} and \cite[Chapter~XII, Lemma~6.6]{Z} for details. Since, for each $0<r<1$, the Maclaurin coefficients of $f_r$ have the same property, we may integrate \eqref{Eq:old-case} as in the proof of Proposition~\ref{inclusion} to deduce the assertion.
\end{proof}

\section{Dualities}\label{section:dualities}

The following lemma describes the dual of the Dirichlet-type space $D^p_{\om}$ when $1<p<\infty$ and $\om\in\DD$.

\begin{lemma}\label{le:dualdp}
Let $1<p<\infty$ and $\om\in\DD$. Then $(D^p_{\om})^\star\simeq D^{p'}_{\om}$ via the pairing
$$
\langle f,g\rangle_{D^2_{\om}}=\langle f',g'\rangle_{A^2_{\om}}+f(0)\overline{g(0)}
$$
with equivalence of norms.
\end{lemma}

\begin{proof}
Let us first show that each $g\in D^{p'}_{\om}$ induces a bounded linear functional on $D^p_{\om}$.
H\"{o}lder's inequality yields
	\begin{equation*}
	\begin{split}
	|\langle f,g\rangle_{D^2_{\om}}|
	&\le\left(\int_{\D}|f'(z)|^p\om(z)\,dA(z)\right)^{\frac{1}{p}}
	\left(\int_{\D}|g'(z)|^{p'}\om(z)\,dA(z)\right)^{\frac{1}{p'}}+|f(0)||g(0)|\\
	&\lesssim\|f\|_{D^p_{\om}}\|g\|_{D^{p'}_{\om}},\quad f,g\in\H(\D).
	\end{split}
	\end{equation*}
Thus each $g\in D^{p'}_{\om}$ induces a bounded linear functional on $D^p_{\om}$ via the $D^2_{\om}$-pairing.

Let now $L$ be a bounded linear functional on $D^p_{\om}$. By the proof of \cite[Theorem~7]{PelRat2020}, we have $(A^p_{\om})^{\star}\simeq A^{p'}_{\om}$ via the $A^2_{\om}$-pairing with equivalence of norms. For each $f\in D^p_{\om}$, there exists $F=F_f\in A^p_{\om}$ such that $\mathcal{I}(F)=f-f(0)$, where $\mathcal{I}(F)(z)=\int_0^zF(\zeta)\,d\zeta$. Further, $\mathcal{I}$ is an isometric mapping from $A^p_{\om}$ to $D^p_{\om}$, in particular, it is bounded. Therefore the composition $L\circ\mathcal{I}$ is a bounded linear functional on $A^p_{\om}$, and hence there exists a unique $G\in A^{p'}_{\om}$ such that $\|G\|_{A^{p'}_{\om}}\lesssim\|L\circ\mathcal{I}\|\lesssim\|L\|$ and
	\begin{equation*}
	\begin{split}
	L(f)=L(f-f(0)+f(0))&=(L\circ\mathcal{I})(F)+f(0)L(1)\\
	&=2\int_{\D}F(z)\overline{G(z)}\om(z)dA(z)+f(0)L(1)\\
	&=2\int_{\D}f'(z)\overline{G(z)}\om(z)dA(z)+f(0)L(1).
	\end{split}
	\end{equation*}
Further, since $G\in A^{p'}_{\om}$, there exists a unique $g\in D^{p'}_{\om}$ such that $g'=G$ and $\overline{g(0)}=L(1)$. Consequently, there exists a unique $g\in D^{p'}_{\om}$ such that
	$$
	L(f)=\int_{\D}f'(z)\overline{g'(z)}\om(z)dA(z)+f(0)\overline{g(0)}=\langle f,g\rangle_{D^2_{\om}}.
	$$
Moreover, $\|g\|^{p'}_{D^{p'}_{\om}}=\|G\|^{p'}_{A^{p'}_{\om}}+|L(1)|^{p'}\lesssim\|L\|^{p'}$, and the assertion is proved.
\end{proof}

The next lemma shows that, for $0<p<1$, the dual of $D^p_{\om}$ can be identified with $\ZZ$ via the $D^2_W$-pairing where $W$ depends on $\om$ appropriately. The definition of $W$ is given in \eqref{Eq:W} below, and the identity $\widehat{W}(z)=\widehat{\om}(z)^{\frac1p}(1-|z|)^{\frac1p-1}$ explains why this choice appears to be convenient in concrete calculations. Observe that recently it was shown in \cite[Lemma~6]{PRW} that $(D^p_{p-1})^\star\simeq\B^2$ via the $A^2_{\frac1p-1}$-pairing, provided $0<p\le1$. The crucial step in the proof of this last-mentioned duality relies on technical tools related to coefficient multipliers of the Bloch spaces. Since the pairing is different in our setting, we can avoid many tedious calculations all together, and the proof itself becomes more straightforward and transparent via a suitable use of a Carleson embedding theorem for Bergman spaces.

\begin{lemma}\label{lemma:dual-Dirichlet}
Let $0<p<1$, $\om\in\DD$ and
	\begin{equation}\label{Eq:W}
	W(z)=W_{p,\om}(z)
	=\left(\frac1p-1\right)\widehat{\om}(z)^{\frac1p}(1-|z|)^{\frac1p-2}
	+\frac{\om(z)}{p}\widehat{\om}(z)^{\frac1p-1}(1-|z|)^{\frac1p-1},\quad z\in\D.
	\end{equation}
Then $(D^p_{\om})^\star\simeq\ZZ$ via the pairing
	\begin{equation}\label{D^1-pairing}
	\langle f,g\rangle_{D^2_{W}}=\langle f',g'\rangle_{A^2_{W}}+f(0)\overline{g(0)}
	\end{equation}
with equivalence of norms.
\end{lemma}

\begin{proof}
Let $f,g\in\H(\D)$. Then Green's formula and Fubini's theorem yield
	\begin{equation*}
	\begin{split}
	|\langle f',g'\rangle_{A^2_{W}}|
	&=\left|\int_\D f'(z)\overline{g'(z)}W(z)\,dA(z)\right|\\
	&=\left|\int_0^1\left(2\int_\D f''(rz)r\overline{g''(rz)r}\log\frac1{|z|}\,dA(z)+f'(0)\overline{g'(0)}\right)
	W(r)r\,dr\right|\\
	&\le2\int_\D|f''(\z)||g''(\z)|\left(\int_{|\z|}^1\log\frac{r}{|\z|}W(r)r\,dr\right)\,dA(\z)+|f'(0)||g'(0)|W(\D).
	\end{split}
	\end{equation*}
where	
	$$
	\int_{|\z|}^1\log\frac{r}{|\z|}W(r)r\,dr
	\le\log\frac{1}{|\z|}\int_{|\z|}^1W(r)r\,dr
	\le\frac{\whw(\z)^{\frac1p}(1-|\z|)^{\frac1p}}{|\z|}
	$$
by the inequality $-\log t\le\frac1t(1-t)$, valid for all $0<t\le1$, and the identity $\widehat{W}(z)=\widehat{\om}(z)^{\frac1p}(1-|z|)^{\frac1p-1}$. It follows that
	\begin{equation*}
	\begin{split}
	|\langle f',g'\rangle_{A^2_{W}}|
	&\lesssim\|g\|_\ZZ\int_\D|f''(\z)|\frac{\whw(\z)^\frac1p(1-|\z|)^{\frac1p-1}}{|\z|}\,dA(\z)+|f'(0)||g'(0)|\\
	&\lesssim\|g\|_\ZZ\int_\D|f''(\z)|\whw(\z)^\frac1p(1-|\z|)^{\frac1p-1}\,dA(\z)+|f'(0)||g'(0)|
	\end{split}
	\end{equation*}
because $M_1(r,f'')$ is non-decreasing. An easy calculation based on the Cauchy formula gives $M_1(r,h')\le4M_1\left(\frac{1+r}{2},h\right)(1-r)^{-1}$ for all $0<r<1$ and $h\in\H(\D)$. An application of this to function $h=f'$ together with the hypothesis $\om\in\DD$  gives
	\begin{equation}\label{vitui}
	|\langle f',g'\rangle_{A^2_{W}}|
	\lesssim\|g\|_\ZZ\int_\D|f'(\z)|\whw(\z)^\frac1p(1-|\z|)^{\frac1p-2}\,dA(\z)+|f'(0)||g'(0)|.
	\end{equation}
If $0<p\le q<\infty$, then, by \cite[Theorem~1(c)]{PR2015}, the Bergman space $A^p_\om$ is continuously embedded into the measure space $L^q_\mu$ if and only if $\mu(S)\lesssim\om(S)^\frac{q}{p}$ for all Carleson squares $S$. Now that $0<p<1$, we deduce
	\begin{equation*}
	\begin{split}
	\int_{S(a)}\whw(z)^{\frac1p}(1-|z|)^{\frac1p-2}\,dA(z)
	\le\whw(a)^{\frac1p}(1-|a|)\int_{|a|}^1(1-r)^{\frac1p-2}\,dr
	\asymp\left(\om(S(a))\right)^{\frac1p},\quad|a|\to1^-.
	\end{split}
	\end{equation*}
This together with \eqref{vitui} yields
\begin{equation}\label{eq:j1}
|\langle f',g'\rangle_{A^2_{W}}|\lesssim\|g\|_\ZZ\|f\|_{D^p_\om},\quad\text{ for all $f,g\in\H(\D)$}.
\end{equation} Hence each $g\in\ZZ$ induces a bounded linear functional on $D^p_{\om}$ via the pairing \eqref{D^1-pairing}.

Let $L\in(D^p_{\om})^\star$, and recall that $\mathcal{I}(F)(z)=\int_0^zF(\zeta)\,d\zeta$. Then $|(L\circ\mathcal{I})(F)|\lesssim\|\mathcal{I}(F)\|_{D^p_{\om}}=\|F\|_{A^p_{\om}}$ for all $F\in A^p_{\om}$.
 Therefore $L\circ\mathcal{I}\in(A^p_{\om})^\star$. It is known by \cite[Theorem 1]{PP2017} that $(A^p_{\om})^\star$ is isomorphic to the Bloch space via the $A^2_{W}$-pairing.
Hence there exists a unique $G\in\B$ such that $\|G\|_\B\lesssim\|L\circ\mathcal{I}\|\lesssim\|L\|$ and $(L\circ\mathcal{I})(F)=\langle F,G\rangle_{A^2_{W}}$ for all $F\in A^p_{\om}$.
Moreover, for each $f\in D^p_{\om}$, there exists $F=F_f\in A^p_{\om}$ such that $\mathcal{I}(F)=f-f(0)$. Therefore
	\begin{equation*}
	\begin{split}
	L(f)&=L(f-f(0)+f(0))=L(f-f(0))+f(0)L(1)\\
	&=L(\mathcal{I}(F))+f(0)L(1)
	=\langle F,G\rangle_{A^2_{W}}+f(0)L(1)\\
	&=\langle f',G\rangle_{A^2_{W}}+f(0)L(1),\quad f\in D^p_{\om}.
	\end{split}
	\end{equation*}
By picking up $g\in\H(\D)$ such that $g'=G$ and $g(0)=\overline{L(1)}$
we deduce $L(f)=\langle f,g\rangle_{D^2_{W}}$ for all $f\in D^p_{\om}$, and $\|g\|_{\ZZ}=\|G\|_\B+|L(1)|\lesssim\|L\|$.
\end{proof}





We next give an identification, useful for our purposes, for the dual space of $D^1_\om$ in the case when $\om\in\DD$. To this end, for a radial weight $\om$, consider the space
	$$
	\BMOA(\infty,\om)=\{f\in\H(\D): \|f\|_{\BMOA(\infty,\om)}=\sup_{0\le r<1}\big(\|f_r\|_{\BMOA}\whw(r)\big)<\infty\}.
	$$
Let now $\BMOA'(\infty,\om)$ denote space of primitives of functions in $\BMOA(\infty,\om)$ endowed with the norm
$$
\|f\|_{\BMOA'(\infty,\om)}=\|f'\|_{\BMOA(\infty,\om)}+|f(0)|.
$$

\begin{lemma}\label{le:duald1omhat}
Let $\om\in\DD$. Then $(D^1_{\om})^\star\simeq \BMOA'(\infty,\om)$ via the pairing
	$$
	\langle f,g\rangle_{D^2_{\om\whw}}=\langle f',g'\rangle_{A^2_{\om\whw}}+f(0)\overline{g(0)}
	$$
with equivalence of norms.
\end{lemma}

\begin{proof}
First observe that $\om\whw\in\DD$, and
	$$
	(\om_x)^2
	\asymp\left(\widehat{\omega}\left(1-\frac{1}{x}\right)\right)^2
	\asymp\widehat{\omega\whw}\left(1-\frac{1}{x}\right)
	\asymp(\omega\whw)_x,\quad x\ge1,
	$$
by Lemma~\ref{Eqdhat}(iii). This together with \cite[Theorem~4]{PelRat2020}
gives $(A^1_\om)^\star\simeq\BMOA(\infty,\om)$ via the $A^2_{\om\whw}$-pairing. Therefore
	\begin{align*}
	|\langle f, g\rangle_{A^2_{\om\whw}}|&=\left|\int_{\D}f(z)\overline{g(z)}\om(z)\whw(z)\,dA(z)\right|\\
	&\lesssim\|g\|_{\BMOA(\infty,\om)}\|f\|_{A^1_\om}, \quad g\in\BMOA(\infty,\om),\quad f\in A^1_{\om},
	\end{align*}
and hence
\begin{align*}
|\langle f,g\rangle_{D^2_{\om\whw}}|
&\le|\langle f',g'\rangle_{A^2_{\om\whw}}|+|f(0)||g(0)|\\
&\lesssim\|g'\|_{\BMOA(\infty,\om)}\|f'\|_{A^1_\om}+|f(0)||g(0)|\\
&\lesssim\|g\|_{\BMOA'(\infty,\om)}\|f\|_{D^1_\om},\quad g\in\BMOA'(\infty,\om),\quad f\in D^1_{\om}.
\end{align*}
Thus each $g\in\BMOA'(\infty,\om)$ induces a bounded linear functional on $D^1_\om$ via the $D^2_{\om\whw}$-pairing.

The proof can be completed by analogous arguments  to those used in the second part of the proof of Lemma~\ref{lemma:dual-Dirichlet}. The only extra ingredient needed is \cite[Theorem 4]{PelRat2020} which states that $(A_{\om}^1)^\star\simeq \BMOA(\infty,\om)$ via $A^2_{\om\whw}$-pairing. Since the details do not give us anything new, we omit them.
\end{proof}

The dual space of the Banach space $\HL_p^{\om}$ with $1<p<\infty$ can be described as follows. The proof is straightforward and hence omitted.

\begin{lemma}\label{le:dualHLp}
Let $1<p<\infty$ and $\om\in\DD$. Then $(\HL^{\om}_p)^\star\simeq \HL^{\om}_{p'}$ via the pairing
	$$
	\langle f,g \rangle_{A^2_\om}
	=\lim_{r\rightarrow1^-}\sum_{n=0}^{\infty}\widehat{f}(n)\overline{\widehat{g}(n)}\om_{2n+1}r^n
	$$
with equivalence of norms.
\end{lemma}

\section{Case $0<p<1$}\label{Sec:p<1}

We begin with the following lemma which concerns the range $0<p\le1$. It proves the equivalence between (ii) and (iii) in Theorem~\ref{intro:0<p<1}, and since $\om_x\asymp\widehat{\om}\left(1-\frac1x\right)$ due Lemma~\ref{Eqdhat}(iii), it also gives an equivalent condition in terms of moments.

\begin{lemma}\label{AuxLemma}
Let $0<p\le1$ and $\om\in\DD$. Then the following statements are equivalent:
\begin{itemize}
\item[(i)] $\displaystyle \sup_{0\leq r<1}\frac{(1-r)^{2-\frac1p}}{\whw(r)^{\frac1p}}<\infty$;
\item[(ii)]$\displaystyle\sup_{0\le r<1}\left((1-r)\sum_{k=1}^{\infty}\frac{r^{2(k-1)}}{\whw(1-\frac1k)^{\frac1p}k^{2-\frac1p}}\right)<\infty$;
\item[(iii)] $I:D^p_{\om}\rightarrow D^1_0$ is bounded.
\end{itemize}
Similarly, the following statements are equivalent:
\begin{itemize}
\item[(i)] $\displaystyle\limsup_{r\to1^-}\frac{(1-r)^{2-\frac1p}}{\whw(r)^{\frac1p}}=0$;
\item[(ii)]$\displaystyle\limsup_{r\to1^-}\left((1-r)\sum_{k=1}^{\infty}\frac{r^{2(k-1)}}{\whw(1-\frac1k)^{\frac1p}k^{2-\frac1p}}\right)=0$;
\item[(iii)] $I:D^p_{\om}\rightarrow D^1_0$ is compact.
\end{itemize}
\end{lemma}

\begin{proof}
We first observe that in both sets of statements (i) and (iii) are equivalent for all $0<p\le1$ by the Carleson embedding theorems \cite[Theorem~2.1]{PR} and \cite[Theorem~3]{PRS2}. We next show that (i) and (ii) are equivalent. Direct calculations show that
	\begin{equation}\label{theq}
	\begin{split}
	\sum_{k=1}^{\infty}\frac{r^{2(k-1)}}{\whw(1-\frac1k)^{\frac1p}k^{2-\frac1p}}&\asymp
	\int^{\infty}_1\frac{r^{2x}}{\whw(1-\frac1x)^{\frac1p}x^{2-\frac1p}}dx\\
	&=\int_0^1\frac{r^{\frac{2}{1-t}}(1-t)^{-\frac1p}}{\whw(t)^{\frac1p}}dt
	=I_1(r)+I_2(r),
	\end{split}
	\end{equation}
where
	$$
	I_1(r)=\int_0^r\frac{r^{\frac{2}{1-t}}(1-t)^{-\frac1p}}{\whw(t)^{\frac1p}}dt \quad\textrm{and}\quad
	I_2(r)=\int_r^1\frac{r^{\frac{2}{1-t}}(1-t)^{-\frac1p}}{\whw(t)^{\frac1p}}dt.
	$$
Moreover,
	\begin{equation}\label{theq1}
	I_1(r)
	\lesssim\frac{(1-r)^{1-\frac1p}}{\whw(r)^{\frac1p}}
	\lesssim\int_r^{\frac{1+r}{2}}\frac{r^{\frac{2}{1-t}}(1-t)^{-\frac1p}}{\whw(t)^{\frac1p}}dt
	\le I_2(r),\quad r\to1^-.
	\end{equation}
Further, by Lemma~\ref{Eqdhat}(ii) and the change of variable $r^\frac2{1-t}=s$, we deduce
\begin{equation}\label{theq2}
\begin{split}
I_2(r)
&\lesssim\left(\frac{(1-r)^{\beta}}{\whw(r)}\right)^{\frac1p}\int_{r}^1r^{\frac{2}{1-t}}(1-t)^{-\frac{1+\beta}p}dt\\
&=\left(\frac{(1-r)^{\beta}}{\whw(r)}\right)^{\frac1p}2^{1-\frac{1+\beta}{p}}\left(\log\frac1r\right)^{1-\frac{1+\beta}{p}}\int_0^{r^\frac{2}{1-r}}\left(\log\frac1s\right)^{\frac{1+\beta}{p}-2}\,ds
\\ & \le
\left(\frac{(1-r)^{\beta}}{\whw(r)}\right)^{\frac1p}2^{1-\frac{1+\beta}{p}}\left(\log\frac1r\right)^{1-\frac{1+\beta}{p}}\int_0^{1}\left(\log\frac1s\right)^{\frac{1+\beta}{p}-2}\,ds
\\
&\asymp\frac{(1-r)^{1-\frac1p}}{\whw(r)^{\frac1p}},\quad r\to1^-.
\end{split}
\end{equation}
By combining \eqref{theq}, \eqref{theq1} and \eqref{theq2}, we obtain
	\begin{equation}\label{pililo}
	\sum_{k=1}^{\infty}\frac{r^{2(k-1)}}{\whw(1-\frac1k)^{\frac1p}k^{2-\frac1p}}\asymp
	\frac{(1-r)^{1-\frac1p}}{\whw(r)^{\frac1p}},\quad r\to1^-.
	\end{equation}
Hence (ii) is equivalent to (iii) for all $0<p\le 1$.
\end{proof}

We next prove the boundedness part for the case $0<p<1$ of Theorem~\ref{intro:0<p<1}. As an immediate consequence we deduce that if $\om\in\DD$ and $p\le\frac12$ then $T(D^p_{\om}, H^{\infty})$ consists of constant functions only.

\begin{theorem}\label{D^pless1}
Let $0<p<1$, $\om\in\DD$ and $g\in H^{\infty}$. Then the following statements are equivalent:
\begin{itemize}
\item[(i)] $T(D^p_{\om}, H^{\infty})$ consists of constant functions only;
\item[(ii)] $\displaystyle \sup_{0\leq r<1}\frac{(1-r)^{2-\frac1p}}{\whw(r)^{\frac1p}}=\infty$;
\item[(iii)]$\displaystyle\sup_{0\le r<1}\left((1-r)\sum_{k=1}^{\infty}\frac{r^{2(k-1)}}{\whw(1-\frac1k)^{\frac1p}k^{2-\frac1p}}\right)=\infty$;
\item[(iv)] $I: D^p_{\om}\rightarrow D^1_0$ is unbounded.
\end{itemize}
\end{theorem}

\begin{proof}
First observe that (ii)-(iv) are equivalent by Lemma~\ref{AuxLemma}. We next show that (i) and (iii) are equivalent. A careful inspection of the proof of \cite[Theorem~1.1]{CPPR} shows that it can be applied to the space $D^p_{\om}$ in the case $0<p<1$ even if it is not a Banach space. By this observation and Lemma~\ref{lemma:dual-Dirichlet}, $T_g: D^p_{\om}\rightarrow H^{\infty}$ is bounded if and only if $\sup_{z\in\D}\|G^{D^2_{W}}_{g,z}\|_{\mathcal{Z}}<\infty$, where
$\overline{G^{D^2_{W}}_{g,z}(w)}=\int_0^zg'(\zeta)\overline{K^{D^2_{W}}_{\zeta}(w)}\,d\zeta$ and
$K^{D^2_{W}}_{\zeta}(w)=1+\sum_{k=1}^{\infty}\frac{\overline{\zeta}^kw^k}{2k^2W_{2k-1}}$ is the reproducing kernel of $D^2_{W}$ at the point $\zeta\in\D$.

Assume (iii) holds. A direct calculation shows that
	\begin{equation*}
	\begin{split}
	\sup_{z\in\D}\|G^{D^2_{W}}_{g,z}\|_{\mathcal{Z}}
	&\ge\sup_{z\in\D}\left(\sup_{w\in\D}(1-|w|^2)|G^{D^2_{W}}_{g,z})^{''}(w)|\right)\\
	&=\sup_{z\in\D}\left(\sup_{w\in\D}(1-|w|^2)\left|\sum_{k=2}^{\infty}\left(\frac{k-1}{2kW_{2k-1}}
	\sum_{n=0}^{\infty}\frac{\overline{\widehat{g}(n+1)}(n+1)\overline{z}^{n+k+1}}{n+k+1}\right)w^{k-2}\right|\right)\\
	&\gtrsim\sup_{z\in\D}(1-|z|^2)\left|\sum_{k=2}^{\infty}\sum_{n=0}^{\infty}
	\frac{(k-1)\overline{\widehat{g}(n+1)}(n+1)|z|^{2k-4}\overline{z}^{n+3}}{2k(n+k+1)W_{2k-1}}\right|.
	\end{split}
	\end{equation*}
Then Fubini's theorem and Hardy's inequality yield
	\begin{equation*}
	\begin{split}
	\sup_{z\in\D}\|G^{D^2_{W}}_{g,z}\|_{\mathcal{Z}}&\gtrsim\sup_{0\leq r<1}(1-r)
	\int_0^{2\pi}\left|\sum_{n=0}^{\infty}\left(\sum_{k=2}^{\infty}\frac{(k-1)r^{2k+n-1}}{2k(n+k+1)W_{2k-1}}\right)
	\overline{\widehat{g}(n+1)}(n+1)e^{i\theta(n+3)}\right|d\theta\\
	&\gtrsim\sup_{0\le r<1}(1-r)\sum_{n=0}^{\infty}\left(\sum_{k=2}^{\infty}
	\frac{r^{2k+n-1}}{(n+k+1)W_{2k-1}}\right)|\overline{\widehat{g}(n+1)}|.
	\end{split}
	\end{equation*}
If $g\in T(D^p_{\om}, H^{\infty})$ is not a constant, then there exists an $N\in\mathbb{N}\cup\{0\}$ such that $\widehat{g}(N+1)\neq0$. Since $W\in\DDD$ by the proof of \cite[Theorem~1]{PP2017}, Lemma~\ref{Eqdhat}(iii) and the identity $\widehat{W}(z)=\widehat{\om}(z)^{\frac1p}(1-|z|)^{\frac1p-1}$ imply $W_{2k-1}\asymp\widehat{\om}\left(1-\frac1k\right)^\frac1pk^{1-\frac1p}$ for all $k\in\N$. Now that $\widehat{g}(N+1)\neq0$, we deduce
	$$
	\sup_{z\in\D}\|G^{D^2_{W}}_{g,z}\|_{\mathcal{Z}}
	\gtrsim\limsup_{0\le r<1}\left((1-r)\sum_{k=1}^{\infty}\frac{r^{2k}}{\whw(1-\frac1k)^{\frac1p}k^{2-\frac1p}}\right)=\infty.
	$$
This contradiction shows that (i) is satisfied.

Conversely assume that (iii) does not hold. We claim that then $T(D^p_{\om}, H^{\infty})$ contains all polynomials. As a matter of fact, if $g(z)=m_n(z)=z^n$ for some $n\in\N$, then the hypothesis $\omega\in\DD$ yields
\begin{equation*}
\begin{split}
\sup_{z\in\D}\|G^{D^2_{W}}_{g,z}\|_{\mathcal{Z}}
&\lesssim\sup_{0\leq r<1}\left((1-r)n\sum_{k=2}^\infty\frac{(k-1)r^{k-2}}{k(n+k)W_{2k-1}}\right)\\
&\asymp\sup_{0\leq r<1}(1-r)\sum_{k=1}^{\infty}\frac{r^{2(k-1)}}{\whw(1-\frac1k)^{\frac1p}k^{2-\frac1p}}<\infty.
\end{split}
\end{equation*}
Thus (i) and (iii) are equivalent, and the proof if complete.
\end{proof}

According to Theorem~\ref{D^pless1}, the boundedness of $I:D^p_{\om}\rightarrow D^1_0$ is equivalent to the statement that $T(D^p_{\om},H^{\infty})$ contains a non-constant function if $0<p<1$. However, such an equivalence is no longer valid for the case $1<p<\infty$. We will give a counterexample after Proposition~\ref{TgHLpom} in Section~\ref{Sec:p>1}, where the range $1<p<\infty$ is systematically studied.

Theorem~\ref{D^pless1} has the following interesting consequence.

\begin{corollary}
Let $0<p<1$, $\omega\in\DD$ and $g\in\H(\D)$, and let $X^{\om}_{p}\in\left\{\HL^{\om}_{p}, A^p_{\om}, D^p_{\widetilde{\om}_{[p]}}\right\}$. Then $T_g: X^\om_{p}\rightarrow H^\infty$ is bounded if and only if $g$ is a constant.
\end{corollary}

\begin{proof}
Since $\omega\in\DD$ by the hypothesis, we have
	$$
	\int_r^1\widetilde{\om}_{[p]}(t)\,dt\asymp\widehat{\om}(r)(1-r)^p,\quad 0\le r<1,
	$$
and hence
	$$
	\frac{(1-r)^{2-\frac1p}}{\left(\int_r^1\widetilde{\om}_{[p]}(t)\,dt\right)^{\frac1p}}\asymp
	\frac{(1-r)^{1-\frac1p}}{\whw(r)^{\frac1p}}\to\infty,\quad r\to1^-.
	$$
Therefore Theorem~\ref{D^pless1} shows that $T_g:D^p_{\widetilde{\om}_{[p]}}\rightarrow H^\infty$ is bounded if and only if $g$ is a constant. Since $D^p_{\widetilde{\om}_{[p]}}\subset A^p_{\om}\subset\HL^{\om}_{p}$ by Proposition~\ref{inclusion}(i), the assertion follows.
\end{proof}

We next prove the counterpart of Theorem~\ref{D^pless1} for compact operators, which covers the compactness part for the case $0<p<1$ of Theorem~\ref{intro:0<p<1}.

\begin{theorem}\label{compactDpless1}
Let $0<p<1$, $\om\in\DD$ and $g\in H^{\infty}$. Then the following statements are equivalent:
\begin{itemize}
\item[(i)] $T_c(D^p_{\om}, H^{\infty})$ consists of constant functions only;
\item[(ii)] $\displaystyle\limsup_{r\to1^-}\frac{(1-r)^{2-\frac1p}}{\whw(r)^{\frac1p}}>0$;
\item[(iii)]$\displaystyle\limsup_{r\to1^-}\left((1-r)\sum_{k=1}^{\infty}\frac{r^{2k}}{\whw(1-\frac1k)^{\frac1p}k^{2-\frac1p}}\right)>0$;
\item[(iv)] $I: D^p_{\om}\rightarrow D^1_0$ is not compact.
\end{itemize}
\end{theorem}

\begin{proof}
We first observe that (ii)--(iv) are equivalent by Lemma~\ref{AuxLemma}. We next prove (ii)$\Rightarrow$(i) by showing that if $T_g:D^p_\om\to H^\infty$ is compact and $g$ is not a constant, then (ii) fails. To see this we first note that by following the proof of \cite[Theorem~2(iii)]{PRW} line by line, with minor modifications, gives
	\begin{equation}\label{prootu}
  \lim_{R\to1^-}\sup_{a,z\in\D}\int_{\D\setminus D(0,R)}|(G_{g,z}^{D^2_W})''(w)|^2(1-|\vp_a(w)|^2)^2\,dA(w)=0,
	\end{equation}
whenever $T_g:D^p_\om\to H^\infty$ is compact. If $g$ is not a constant, then there exists an $N\in\mathbb{N}\cup\{0\}$ such that $\widehat{g}(N+1)\neq0$. Therefore, by using the first part of the proof of Theorem~\ref{D^pless1}, we obtain
	\begin{equation}\label{prootu2}
	\begin{split}\allowdisplaybreaks
	\sup_{a,z\in\D}&\int_{\D\setminus D(0,R)}|(G_{g,z}^{D^2_W})''(w)|^2(1-|\vp_a(w)|^2)^2\,dA(w)\\
	&\ge\sup_{z\in\D}(1-|z|)^2\int_{\D\setminus D(0,R)}\left|\frac{(G^{D^2_{W}}_{g,z})''(w)}{(1-z\overline{w})^2}\right|^2(1-|w|)^2 dA(w)\\
	&=\sup_{z\in\D}\left((1-|z|)^2\int_{\D\setminus D(0,R)}\left|\left(\sum_{k=2}^{\infty}\frac{k-1}{2kW_{2k-1}}\left(\sum_{n=0}^\infty
	\frac{\widehat{g}(n+1)(n+1)z^{n+k+1}}{n+k+1}\right)\overline{w}^{k-2}\right)\right.\right.\\
	&\qquad\cdot\left.\left.\left(\sum_{j=0}^\infty(j+1)z^j\overline{w}^j\right)\right|^2(1-|w|)^2\,dA(w)\right)\\
	&=\sup_{z\in\D}\left((1-|z|)^2\int_{\D\setminus D(0,R)}(1-|w|)^2\right.\\
	&\qquad\cdot\left.\left|\sum_{m=0}^{\infty}\left(\sum_{k=0}^{m}\frac{(k+1)(m-k+1)}{2(k+2)W_{2k+3}}
	\left(\sum_{n=0}^\infty\frac{\widehat{g}(n+1)(n+1)z^{n+k+3}}{n+k+3}\right)z^{m-k}\right)\overline{w}^m\right|^2\,dA(w)\right)\\
	&\asymp\sup_{z\in\D}\left((1-|z|)^2\sum_{m=0}^\infty\left|\sum_{k=0}^{m}\frac{(k+1)(m-k+1)}{2(k+2)W_{2k+3}}
	\left(\sum_{n=0}^\infty\frac{\widehat{g}(n+1)(n+1)z^{n+m+3}}{n+k+3}\right)\right|^2\right.\\
	&\qquad\cdot\left.\left(\int_R^1s^{2m+1}(1-s)^2\,ds\right)\right)\\
	&\gtrsim\sup_{0\leq r<1}\left((1-r)^2\sum_{m=0}^\infty \left(r^{2m+6}\int_R^1s^{2m+1}(1-s)^2\,ds\right)\right.\\
	&\qquad\cdot\left.\int_0^{2\pi}\left|\sum_{n=0}^{\infty}\widehat{g}(n+1)(n+1)
	\left(\sum_{k=0}^m\frac{(k+1)(m-k+1)}{2(k+2)(n+k+3)W_{2k+3}}\right)r^ne^{i\theta}\right|^2\,d\theta\right)\\
	&\ge|\widehat{g}(N+1)|^2(N+1)^2\sup_{0\leq r<1}(1-r)^2 r^{2N+5}\sum_{m=0}^\infty\left(\int_R^1(rs)^{2m+1}(1-s)^2\,ds\right)
	\left(\sum_{k=0}^{m}\frac{m-k+1}{(k+3)(W_{2k+3})}\right)^2\\
	&\gtrsim (1-R)^2 R^{2N+5}\int_R^1(1-s)^2\left(\sum_{m=0}^\infty(Rs)^{2m+1}
	\left(\sum_{k=0}^{m}\frac{m-k+1}{(k+3)(W_{2k+3})}\right)^2\right)\,ds.
	\end{split}
	\end{equation}
Now, a change of variables and Lemma~\ref{Eqdhat}(iii) applied to $W\in\DDD$ yield
	\begin{equation*}
	\begin{split}
	\sum_{m=0}^\infty(Rs)^{2m+1}\left(\sum_{k=0}^{m}\frac{m-k+1}{(k+3)(W_{2k+3})}\right)^2\,ds
	&\asymp\int_1^\infty (Rs)^{2x}\left(\int_1^x\frac{x-y}{yW_y}\,dy\right)^2\,dx\\
	&=\int_0^1\frac{(Rs)^{\frac{2}{1-r}}}{(1-r)^4}\left(\int_0^r\frac{r-t}{(1-t)^2\widehat{W}(t)}\,dt\right)^2\,dr.
	\end{split}
	\end{equation*}
By Lemma~\ref{Eqdhat}(ii), there exists a $\beta=\beta(W)>0$ such that $\frac{\widehat{W}(t)}{(1-t)^\beta}\lesssim\frac{\widehat{W}(r)}{(1-r)^\beta}$ whenever $0\leq t\leq r<1$. This together with the similar estimates as \eqref{theq}--\eqref{theq2} yields
	\begin{equation}\label{prootu3}
	\begin{split}
	&\sum_{m=0}^\infty(Rs)^{2m+1}\left(\sum_{k=0}^{m}\frac{m-k+1}{(k+3)(W_{2k+3})}\right)^2\,ds\\
	&\quad\gtrsim\int_0^1\frac{(Rs)^{\frac{2}{1-r}}}{(1-r)^{4-2\beta}\widehat{W}(r)^2}\left(\int_0^r\frac{r-t}{(1-t)^{2+\beta}}\,dt\right)^2\,dr\\
	&\quad\asymp \int_0^1\frac{(Rs)^{\frac{2}{1-r}}}{(1-r)^{4}\widehat{W}(r)^2}\,dr
	\asymp\frac{1}{(1-Rs)^3\widehat{W}(Rs)^2},\quad \frac12\le R,s<1.
	\end{split}
	\end{equation}
Therefore, by combining \eqref{prootu}--\eqref{prootu3}, and then applying Lemma~\ref{Eqdhat}(ii) and the identity $\widehat{W}(z)=\widehat{\om}(z)^{\frac1p}(1-|z|)^{\frac1p-1}$ we finally obtain
	\begin{align*}
	0&=\lim_{R\rightarrow1^-}\sup_{a,z\in\D}\int_{\D\setminus D(0,R)}|(G_{g,z}^{D^2_W})''(w)|^2(1-|\vp_a(z)|^2)^2\,dA(w)\\
	&\gtrsim \lim_{R\rightarrow1^-}R^{2N+5}(1-R)^2\int_R^1\frac{(1-s)^2}{(1-Rs)^3\widehat{W}(Rs)^2}\,ds\\
	&\asymp\lim_{R\rightarrow1^-}\frac{(1-R)^2}{\widehat{W}(R)^2}
	=\lim_{R\rightarrow1^-}\left(\frac{(1-R)^{2-\frac1p}}{\whw(R)^{\frac1p}}\right)^2,
	\end{align*}
which contradicts (ii). Thus (ii) implies (i).

We complete the proof by showing that (i) implies (iv). This implication is established by proving that $T_g:D^p_\om\to H^\infty$ is compact, whenever $I:D^p_\om\to D^1_0$ is compact and $g'\in H^\infty$. Observe that,
	\begin{equation}\label{Eq:Pommerenke-lemma}
	\int_0^rM_\infty(t,f)\,dt\le\pi rM_1(r,f),\quad f\in\H(\D),\quad 0<r<1,
	\end{equation}
by \cite[Hilfssatz~1]{Po61/62}. Further,
    \begin{equation*}
    \begin{split}
    |f(re^{i\t})|
    =\left|\int_0^{re^{i\t}}f'(\z)\,d\z+f(0)\right|
    \le2\left(\int_0^r|f'(te^{i\t})|\,dt+|f(0)|\right),\quad f\in\H(\D),
    \end{split}
    \end{equation*}
and hence
	\begin{equation}\label{pili}
	M_\infty(r,f)\lesssim\int_0^r M_\infty(t,f')\,dt+|f(0)|,\quad f\in\H(\D).
	\end{equation}
By combining \eqref{Eq:Pommerenke-lemma} and \eqref{pili} we deduce
	\begin{equation*}
  \begin{split}
	\|T_g(f)\|_{H^\infty}
	&\le\|g'\|_{H^\infty}\int_0^1M_\infty(r,f)\,dr
	\lesssim\|g'\|_{H^\infty}\left(\int_0^1\left(\int_0^r M_\infty(t,f')\,dt\right)dr+|f(0)|\right)\\
	&\lesssim\|g'\|_{H^\infty}\|f\|_{D^1_0},
	\end{split}
  \end{equation*}
and thus $T_g:D^1_0\to H^\infty$ is bounded and $\|T_g\|_{D^1_0\to H^\infty}\lesssim\|g'\|_{H^\infty}$. Since $I:D^p_\om\to D^1_0$ was assumed to be compact, $T_c(D^p_\om,H^\infty)$ contains each $g\in\H(\D)$ such that $g'\in H^\infty$.
\end{proof}

In the case when the Maclaurin coefficients of the symbol are non-negative, we have the following result which establishes the statement in Theorem~\ref{nonnegativepless1intro}(i).

\begin{proposition}\label{nonnegativepless1}
Let $0<p<1$ and $\omega\in\DD$, and let $g\in H^\infty$ such that $\widehat{g}(n)\geq0$ for all $n\in\mathbb{N}\cup\{0\}$. Then $T_g:D^p_{\om}\to H^\infty$ is bounded if and only if
	$$
	\sup_{0\leq r<1}\left((1-r)\sum_{k=0}^{\infty}\frac{(k+1)^{\frac1p-1}r^k}{(\om_k)^{\frac1p}}\left(\sum_{n=0}^{\infty}
	\frac{\widehat{g}(n+1)(n+1)}{n+k+1}\right)\right)<\infty,
	$$
and
	$$
	\|T_g\|_{D^p_{\om}\rightarrow H^\infty}
	\asymp\sup_{0\leq r<1}\left((1-r)\sum_{k=0}^{\infty}\frac{(k+1)^{\frac1p-1}r^k}{(\om_k)^{\frac1p}}\left(\sum_{n=0}^{\infty}
	\frac{\widehat{g}(n+1)(n+1)}{n+k+1}\right)\right).
	$$
Moreover, $T_g:D^p_{\om}\to H^\infty$ is compact if and only if
	\begin{equation}\label{cptcondition}
	\limsup_{r\rightarrow1^-}\left((1-r)\sum_{k=0}^{\infty}\frac{(k+1)^{\frac1p-1}r^k}{(\om_k)^{\frac1p}}\left(\sum_{n=0}^{\infty}
	\frac{\widehat{g}(n+1)(n+1)}{n+k+1}\right)\right)=0.
	\end{equation}
\end{proposition}

\begin{proof} The first part of the proof of Theorem~\ref{D^pless1}, with minor modifications, show that
	$$
	\|T_g\|_{D^p_{\om}\rightarrow H^\infty}
	\asymp\sup_{z\in\D}\|G_{g,z}^{D^2_W}\|_{\ZZ}
	\asymp\sup_{0\le r<1}(1-r)\sum_{k=0}^{\infty}\frac{r^k}{W_{2k+1}}\left(\sum_{n=0}^{\infty}
	\frac{\widehat{g}(n+1)(n+1)}{n+k+1}\right).
	$$
Since $\om,W\in\DD$, Lemma~\ref{Eqdhat}(ii)(iii) yields
	$$
	W_{2k-1}
	\asymp W_{k+1}
	\asymp\widehat{W}\left(1-\frac{1}{k+1}\right)
	\asymp\widehat{\om}\left(1-\frac{1}{k+1}\right)^\frac1p(k+1)^{1-\frac1p}
	\asymp(\omega_k)^\frac1p(k+1)^{1-\frac1p},\quad k\in\N.
	$$
The statements concerning the boundedness are now proved.

To verify the assertion on the compactness, assume first that $T_g:D^p_{\om}\rightarrow H^\infty$ is compact. Let $(H^\infty)^\star$ be the identification of the dual of $H^\infty$ via the $D^2_W$-pairing. Then $T_g^*:(H^\infty)^\star\rightarrow(D^p_\om)^\star$ is compact. Moreover, by \cite[(2.4)]{CPPR},
	$$
	\langle T_g(f),K_z^{D^2_W}\rangle_{D^2_W}
	=T_g(f)(z)
	=\langle f,G^{D^2_W}_{g,z}\rangle_{D^2_W},\quad z\in\D.
	$$
and hence $T_g^*(K_z^{D^2_W})=G^{D^2_W}_{g,z}$ and $\|K_z^{D^2_W}\|_{(H^\infty)^\star}\le1$ for all $z\in\D$. Therefore $\{G^{D^2_W}_{g,z}:z\in\D\}$ is relatively compact in $\mathcal{Z}$ by Lemma \ref{lemma:dual-Dirichlet}. Hence, for a given $\varepsilon>0$, there exist $z_1,z_2,\ldots,z_N\in\D$ such that for each $z\in\D$, we have $\|G^{D^2_W}_{g,z}-G^{D^2_W}_{g,z_j}\|_{\mathcal{Z}}<\varepsilon$ for some $j=j(z)\in\{1,\ldots,N\}$. This together with the fact that $\lim_{|w|\to 1^-}|(G^{D^2_W}_{g,z_j})^{\prime\prime}(w)|(1-|w|)=0$ for each $j\in\{1,\dots N\}$, implies that
	\begin{align*}
	\lim_{|w|\rightarrow1^-}\sup_{z\in\D}|(G^{D^2_W}_{g,z})^{\prime\prime}(w)|(1-|w|)=0.
	\end{align*}
Consequently, by the hypothesis $\widehat{g}(n)\ge0$ for all $n\in\mathbb{N}\cup\{0\}$ and the first part of the proof, we have
	\begin{equation*}
	\begin{split}
	0&=\limsup_{r\rightarrow1^-}\sup_{s\in (0,1)}|(G^{D^2_W}_{g,s})^{\prime\prime}(r)|(1-r)\\
	&\asymp\limsup_{r\rightarrow1^-}\left((1-r)\sum_{k=0}^{\infty}\frac{k^{\frac1p-1}r^k}{(\om_k)^{\frac1p}}\left(\sum_{n=0}^{\infty}
	\frac{\widehat{g}(n+1)(n+1)}{n+k+1}\right)\right).
	\end{split}
	\end{equation*}

Conversely, assume that \eqref{cptcondition} holds. To prove the compactness of $T_g: D^p_\om\rightarrow H^\infty$, it suffices to show that each norm bounded family $\{f_n\}$ in $D^p_\om$ such that $f_n\rightarrow0$ uniformly on compact subsets of $\D$ satisfies $\|T_g(f_n)\|_{H^\infty}\rightarrow0$, as $n\to\infty$. Let now $\{f_n\}$ be such a family. An argument similar to that employed in the proof of Lemma~\ref{lemma:dual-Dirichlet} yields
	\begin{equation}\label{jutskoi}
	\begin{split}
	\|T_g(f_n)\|_{H^\infty}
	&=\sup_{z\in\D}|\langle f_n, G^{D^2_W}_{g,z}\rangle_{D^2_W}|\\
	&=\sup_{z\in\D}\left(\left|\int_\D f'_n(\zeta)(G^{D^2_W}_{g,z})'(\zeta)W(\zeta)\,dA(\zeta)\right|+|G^{D^2_W}_{g,z}(0)||f_n(0)|\right)\\
	&\lesssim\sup_{z\in\D}\left(\int_\D|f''_n(\zeta)||(G^{D^2_W}_{g,z})''(\zeta)|\frac{(1-|\zeta|)}{|\zeta|}\widehat{W}(\zeta)\,dA(\zeta)\right.\\
	&\quad\left.+|(G^{D^2_W}_{g,z})'(0)||(f_n)'(0)|+|G^{D^2_W}_{g,z}(0)||f_n(0)|\right).
	\end{split}
	\end{equation}
By the uniform convergence we may choose $N=N(\varepsilon,R)\in\mathbb{N}$ such that $\max\{|f_n(0)|,|f'_n(0)|,|f^{''}_n(\xi)|\}<\varepsilon$ for all $n\ge N$ and $\xi\in\overline{D(0,R)}$. Therefore the first part of the proof concerning the boundedness, and\eqref{cptcondition} imply
	\begin{equation}\label{Eq:cpt1}
	\sup_{z\in\D}(|(G^{D^2_W}_{g,z})'(0)||(f_n)'(0)|
	+|G^{D^2_W}_{g,z}(0)||f_n(0)|)
	\lesssim\varepsilon\sup_{z\in\D}\|G^{D^2_W}_{g,z}\|_{\mathcal{Z}}
	\lesssim\varepsilon,\quad n\ge N.
	\end{equation}	
	
Further, by \eqref{cptcondition}, for each $\varepsilon>0$ there exists an $R=R(\varepsilon)\in(0,1)$ such that $|(G^{D^2_W}_{g,z})^{\prime\prime}(z)|(1-|z|)<\varepsilon$ for $|z|>R$. Since $\|f'_n\|_{D^1_{\widehat{W}}}\lesssim\|f_n\|_{D^p_\om}$ by the proof of Lemma~\ref{lemma:dual-Dirichlet}, we deduce
	\begin{equation}\label{Eq:cpt2}
	\begin{split}
	&\int_\D|f''_n(\zeta)|\left|(G^{D^2_W}_{g,z})''(\zeta)\right|\frac{(1-|\zeta|)}{|\zeta|}\widehat{W}(\zeta)\,dA(\zeta)\\
	&\quad=\left(\int_{D(0,R)}
	\qquad+\int_{\D\setminus D(0,R)}\right)|f''_n(\zeta)|\left|(G^{D^2_W}_{g,z})''(\zeta)\right|\frac{(1-|\zeta|)}{|\zeta|}\widehat{W}(\zeta)\,dA(\zeta)\\	
	&\quad\lesssim\varepsilon\left(\sup_{z\in\D}\left\|G^{D^2_W}_{g,z}\right\|_{\mathcal{Z}}\right)
	+\varepsilon\int_{\D}|f_n''(\zeta)|\widehat{W}(\zeta)\,dA(\zeta)\\
	&\quad\lesssim\varepsilon(1+\|f_n\|_{D^p_\om})\lesssim\varepsilon,\quad n\ge N.
	\end{split}
	\end{equation}
By combining \eqref{jutskoi}--\eqref{Eq:cpt2} we deduce $\lim_{n\rightarrow\infty}\|T_gf_n\|_{H^\infty}=0$.
\end{proof}

Theorem~\ref{D^pless1} shows that $\frac12<p<1$ is a necessary condition for $T(D^p_\om,H^\infty)$ to be nontrivial (provided $0<p<1$). In this case there exist weights $\om\in\DDD$ and symbols $g$ such that $T_g:D^p_\om\rightarrow H^\infty$ is bounded but not compact. As a matter of fact the standard weight $\om(z)=(1-|z|)^{2p-2}$ satisfies $\limsup_{r\rightarrow1^-}\frac{(1-r)^{2-\frac1p}}{\whw(r)^{\frac1p}}=\frac{1}{\sqrt[p]{2p-1}}>0$, and thus $T_c(D^p_\om, H^\infty)$ consists of constant functions only by Theorem~\ref{compactDpless1}. However, it is easy to see that in this case $T(D^p_\om, H^\infty)$ contains all polynomials.

\section{Case $p=1$}
We begin with the following result which covers the boundedness part for the case $p=1$ of Theorem~\ref{intro:0<p<1}.

\begin{theorem}\label{th:p=1}
Let $\omega\in\DD$ and $g\in\H(\D)$. Then the following statements are equivalent:
\begin{itemize}
\item[(i)] $T(D^1_\om,H^\infty)$ consists of constant functions only;
\item[(ii)] $\sup_{0\le r<1} \frac{(1-r)}{\whw(r)}=\infty$;
\item[(iii)] $\sup_{0\le r<1}\left((1-r)\sum_{k=1}^\infty\frac{r^k}{(k+1)\om_k}\right)=\infty$;
\item[(iv)] $I:D^1_\om\to D^1_0$ is unbounded.
\end{itemize}
\end{theorem}

\begin{proof}
First observe that (ii)--(iv) are equivalent by Lemma~\ref{AuxLemma}. We show next that (ii) implies (i). Assume (ii), and suppose on the contrary to (i) that there exists a $g\in T(D^1_\om,H^\infty)$ such that $\widehat{g}(N+1)\ne0$ for some $N\in\N\cup\{0\}$. Then \cite[Theorem~1.1]{CPPR} and Lemma~\ref{le:duald1omhat} show that $g\in T(D^1_\om,H^\infty)$ if and only if $\sup_{z\in\D}\|G^{D^2_{\om\whw}}_{g,z}\|_{\BMOA'(\infty,\om)}<\infty$,  where
$\overline{G^{D^2_{\om\whw}}_{g,z}(w)}=\int_0^zg'(\zeta)\overline{K^{D^2_{\om\whw}}_{\zeta}(w)}\,d\zeta$ and
$K^{D^2_{\om\whw}}_{\zeta}(w)=1+\sum_{k=1}^{\infty}\frac{\overline{\zeta}^kw^k}{2k^2(\om\whw)_{2k-1}}$ is the reproducing kernel of $D^2_{\om\whw}$ associated with the point $\zeta\in\D$.  A simple computation shows that
	$$
	\overline{G^{D^2_{\om\whw}}_{g,z}(w)}
	=\sum_{n=0}^\infty\widehat{g}(n+1)z^{n+1}+
	\sum_{k=1}^{\infty}\frac{1}{2k^2(\om\whw)_{2k-1}}\left(\sum_{n=0}^\infty\frac{\widehat{g}(n+1)(n+1)z^{n+k+1}}{n+k+1}\right)\overline{w}^k.
	$$
Therefore
	\begin{equation*}
	\begin{split}
	\infty&>\sup_{z\in\D}\|G^{D^2_{\om\whw}}_{g,z}\|^2_{\BMOA'(\infty,\om)}
	\ge\sup_{z\in\D}\|(G^{D^2_{\om\whw}}_{g,z})^{\prime}\|^2_{\BMOA(\infty,\om)}\\
	&\asymp\sup_{z\in\D}\sup_{0<R<1}
	\left(\whw(R)^2\sup_{a\in\D}\int_{\D}|(G^{D^2_{\om\whw}}_{g,z})^{\prime\prime}(Rw)R|^2(1-|\varphi_a(w)|^2)\,dA(w)\right)\\
	&\ge\sup_{z\in\D}\sup_{0<R<1}
	\left(\whw(R)^2(1-|z|)\int_{\D}\left|\left(\sum_{k=0}^\infty\frac{R^{k+2}(k+1)}{2(k+2)(\om\whw)_{2k+3}}
\left(\sum_{n=0}^\infty\frac{\widehat{g}(n+1)(n+1)z^{n+k+3}}{n+k+3}\right)\overline{w}^k\right)\right.\right.\\
	&\quad\cdot\left.\left.\left(\sum_{j=0}^\infty z^j\overline{w}^j\right)\right|^2(1-|w|)\,dA(w)\right)\\
	&=\sup_{z\in\D}\sup_{0<R<1}\Bigg(\whw(R)^2\Bigg((1-|z|)\\
	&\quad\cdot\left.\int_\D\left|\sum_{m=0}^\infty\sum_{k=0}^m\frac{R^{k+2}(k+1)}{2(k+2)(\om\whw)_{2k+3}}
\left(\sum_{n=0}^\infty\frac{\widehat{g}(n+1)(n+1)z^{n+m+3}}{n+k+3}\right)\overline{w}^m\right|^2(1-|w|)\,dA(w)\Bigg)\right)\\
	&\asymp\sup_{z\in\D}\sup_{0<R<1}
	\left(\whw(R)^2\sum_{m=0}^\infty\frac{(1-|z|)}{(m+1)^2}\left|\sum_{k=0}^m\frac{R^{k+2}(k+1)}{(k+2)(\om\whw)_{2k+3}}
	\left(\sum_{n=0}^\infty\frac{\widehat{g}(n+1)(n+1)z^{n+m+3}}{n+k+3}\right)\right|^2\right)\\
	&\gtrsim\sup_{0<t<1}\sup_{0<R<1}
	\left(\whw(R)^2(1-t)\left(\sum_{m=0}^\infty\frac{t^{2m+6}}{(m+1)^2}\right.\right.\\
&\quad\cdot\left.\left.\int_0^{2\pi}\left|\sum_{n=0}^\infty\widehat{g}(n+1)(n+1)
\left(\sum_{k=0}^m\frac{R^{k+2}(k+1)}{2(n+k+3)(k+2)(\om\whw)_{2k+3}}\right)t^ne^{i\theta}\right|^2\,d\theta\right)\right)\\
&\gtrsim\sup_{0<t<1}\sup_{0<R<1}\left(\whw(R)^2(1-t)t^{2N+6}\sum_{m=0}^\infty\frac{t^{2m}}{(m+1)^2}\left(\sum_{k=0}^m\frac{R^{k+2}}{(k+1)(\om\whw)_{2k+3}}\right)^2\right).
\end{split}
\end{equation*}
Since $\om\in\DD$, we have $\om\whw\in\DD$.
This fact and Lemma~\ref{Eqdhat}(ii) together with standard arguments yield
	\begin{equation}\label{poiu}
	\begin{split}
	\sum_{m=0}^\infty\frac{t^{2m}}{(m+1)^2}\left(\sum_{k=0}^m\frac{R^k}{(k+1)(\om\whw)_{2k+3}}\right)^2
	&\asymp\int_1^\infty\frac{t^{2x}}{x^2}\left(\int_1^x\frac{R^y}{y(\om\whw)_y}\,dy\right)^2\,dx\\
	&=\int_0^1t^{\frac{2}{1-r}}\left(\int_0^r\frac{R^{\frac{1}{1-s}}}{\widehat{\om\whw}(s)(1-s)}\,ds\right)^2\,dr\\
	&\gtrsim\int_0^1(Rt)^{\frac{2}{1-r}}\frac{dr}{{\whw}(r)^4}\\
	&\gtrsim\frac{(1-Rt)}{{\whw}(Rt)^4},\quad \frac12\le t,R<1.
	\end{split}
	\end{equation}
Therefore
	\begin{align*}
	\infty&>\sup_{z\in\D}\|G^{D^2_{\om\whw}}_{g,z}\|_{\BMOA'(\infty,\om)}\\
	&\gtrsim\sup_{\frac12<t<1}\sup_{\frac12<R<1}\left(\whw(R)^2(1-t)t^{2N+6}\sum_{m=0}^\infty\frac{t^{2m}}{(m+1)^2}\left(\sum_{k=0}^m\frac{R^{k+2}}{(k+1)(\om\whw)_{2k+3}}\right)^2\right)\\
	&\gtrsim\sup_{\frac12<t<1}\sup_{\frac12<R<1}\whw(R)^2(1-t)t^{2N+6}\frac{(1-Rt)}{{\whw}(Rt)^4}\\
	&\gtrsim\sup_{\frac12<t<1}\left(\frac{1-t}{\whw(t)}\right)^2
	\asymp\sup_{0\le t<1}\left(\frac{1-t}{\whw(t)}\right)^2.
	\end{align*}
This contradicts (ii). Thus we have shown that (ii) implies (i).

The last part of the proof of Theorem~\ref{compactDpless1} shows that if (iv) is not satisfied, that is, $D^1_\om$ is continuously embedded into $D^1_0$, then $T(D^1_\om,H^\infty)$ contains all analytic functions with bounded derivative. This observation implies that (iv) follows from (i), and completes the proof of the theorem.
\end{proof}

Theorem~\ref{th:p=1} has the following interesting consequence.

\begin{corollary}
Let $\omega\in\DD$, $g\in\H(\D)$ and $X^{\om}_{1}\in\{\HL^{\om}_{1}, A^1_{\om}, D^1_{\widetilde{\om}_1}\}$. Then $T(X^{\om}_{1},H^\infty)$ consists of constant functions only.
\end{corollary}

\begin{proof}
Since
	$$
	\frac{(1-r)}{\widehat{\whw}(r)}\asymp\frac{1}{\whw(r)}\to\infty,\quad r\to1^-,
	$$
the assertion follows by Proposition~\ref{inclusion} and Theorem~\ref{th:p=1}.
\end{proof}

Next, we prove the compact version of Theorem~\ref{th:p=1}, and hence prove the compactness part for the case $p=1$ of Theorem~ \ref{intro:0<p<1}.
\begin{theorem}\label{compactDp=1}
Let $\omega\in\DD$ and $g\in\H(\D)$. Then the following statements are equivalent:
\begin{itemize}
\item[(i)] $T_c(D^1_{\om}, H^{\infty})$ consists of constant functions only;
\item[(ii)] $\displaystyle\limsup_{r\to1^-}\frac{(1-r)}{\whw(r)}>0$;
\item[(iii)] $\displaystyle\limsup_{r\to1^-}(1-r)\sum_{k=1}^\infty\frac{r^k}{(k+1)\om_k}>0$;
\item[(iv)] $I: D^1_{\om}\rightarrow D^1_0$ is not compact.
\end{itemize}
\end{theorem}

\begin{proof}
Lemma~\ref{AuxLemma} and the proof of Theorem~\ref{compactDpless1} show that (ii)--(iv) are equivalent, and (i) implies (iv). Therefore it remains to show that (ii) implies (i).

Assume (ii) and suppose on the contrary to (i) that there exists a $g\in T_c(D^1_\om,H^\infty)$ such that $\widehat{g}(N+1)\ne0$ for some $N\in\N\cup\{0\}$. First observe that an argument similar to that used in the proof of \cite[Theorem~2(iii)]{PRW} gives
	\begin{equation}\label{eqp=1a}
	\lim_{R\rightarrow1^-}\sup_{a,z\in\D}\sup_{0<r<1}\left(\whw^2(r)\int_{\D\setminus D(0,R)}|(G^{D^2_{\om\whw}}_{g,z})^{\prime\prime}(rw)r|^2(1-|		\varphi_a(w)|)\,dA(w)\right)=0
	\end{equation}
for each $g\in T_c(D^1_\om, H^\infty)$.

Arguing as in the proof of Theorem~\ref{th:p=1} we deduce
	\begin{equation*}
	\begin{split}
	&\quad\sup_{a,z\in\D}\sup_{0<r<1}\left(\whw(r)^2\int_{\D\setminus D(0,R)}|(G^{D^2_{\om\whw}}_{g,z})^{\prime\prime}(rw)r|^2(1-|\varphi_a(w)|)\,dA(w)\right)\\
	&\gtrsim\sup_{z\in\D}\sup_{0<r<1}\left(\whw(r)^2(1-|z|)\sum_{m=0}^\infty\int_R^1s^{2m+1}(1-s)\,ds\left|\sum_{k=0}^m\frac{r^{k+2}(k+1)}{2(k+2)(\om\whw)_{2k+3}}\right.\right.\\
	&\quad\cdot\left.\left.\left(\sum_{n=0}^\infty\frac{\widehat{g}(n+1)(n+1)z^{n+m+3}}{n+k+3}\right)\right|^2\right)\\
	&\gtrsim\sup_{0<t<1}\sup_{0<r<1}\left(\whw(r)^2(1-t)t^{2N+6}\int_R^1(1-s)\left(
\sum_{m=0}^\infty(st)^{2m}\left(\sum_{k=0}^m\frac{r^k}{(k+1)(\om\whw)_{2k+3}}\right)^2\right)\,ds\right)\\
	&\ge\sup_{0<t<1}\sup_{0<r<1}\left(\whw(r)^2(1-t)t^{2N+6}\int_R^1(1-s)\left(
	\sum_{m=0}^\infty(rst)^{2m}\left(\sum_{k=0}^m\frac{1}{(k+1)(\om\whw)_{2k+3}}\right)^2\right)\,ds\right).
	\end{split}
	\end{equation*}
By  \eqref{poiu} we obtain
	\begin{equation*}
	\sum_{m=0}^\infty(rst)^{2m}\left(\sum_{k=0}^m\frac{1}{(k+1)(\om\whw)_{2k+3}}\right)^2
	\gtrsim\frac{1}{(1-rst)\whw(rst)^4},\quad \frac12< r,s,t<1.
	\end{equation*}
The above estimates yield
	\begin{align*}
	&\quad\limsup_{R\rightarrow1^-}\sup_{a,z\in\D}\sup_{0<r<1}\left(\whw(r)^2\int_{\D\setminus D(0,R)}|(G^{D^2_{\om\whw}}_{g,z})^{\prime\prime}(rw)r|^2(1-|\varphi_a(w)|)\,dA(w)\right)\\
	&\gtrsim\limsup_{R\rightarrow1^-}\sup_{\frac12<t<1}\sup_{\frac12<r<1}\left(\whw(r)^2(1-t)t^{2N+6}\int_R^1\frac{1-s}{(1-rst)\whw(rst)^4}\,ds\right)\\
	&\gtrsim\limsup_{R\rightarrow1^-}\left(\whw(R)^2(1-R)\frac{1-R}{\whw(R)^4}\right)
	=\limsup_{R\rightarrow1^-}\left(\frac{1-R}{\whw(R)}\right)^2>0.
	\end{align*}
This contradiction with \eqref{eqp=1a} finishes the proof.
\end{proof}

Now we prove Theorem~\ref{nonnegativepless1intro}(ii).

\begin{proposition}
Let $\omega\in\DD$, and let $g\in H^\infty$ such that $\widehat{g}(n)\geq0$ for all $n\in\mathbb{N}\cup\{0\}$. Then $T_g:D^1_{\om}\to H^\infty$ is bounded if and only if
	$$
	\sup_{0<r,s<1}\left(\whw(r)^2\sum_{m=0}^\infty\frac{1-s}{(m+1)^2}\left(\sum_{k=0}^m\frac{r^{k+2}}
	{(\om\whw)_{k+1}}\sum_{n=0}^\infty\frac{\widehat{g}(n+1)(n+1)}{n+k+1}s^{n+m+3}\right)^2\right)<\infty,
	$$
and
	$$
	\|T_g\|_{D^1_{\om}\rightarrow H^\infty}^2
	\asymp\sup_{0<r,s<1}\left(\whw(r)^2\sum_{m=0}^\infty\frac{1-s}{(m+1)^2}
	\left(\sum_{k=0}^m\frac{r^{k+2}}{(\om\whw)_{k+1}}\sum_{n=0}^\infty\frac{\widehat{g}(n+1)(n+1)}{n+k+1}s^{n+m+3}\right)^2\right).
	$$
Moreover, $T_g:D^1_{\om}\to H^\infty$ is compact if and only if
\begin{equation}\label{eqcptp=1}
\limsup_{r\rightarrow1^-}\limsup_{s\rightarrow1^-}\left(\whw(r)^2\sum_{m=0}^\infty\frac{1-s}{(m+1)^2}\left(\sum_{k=0}^m\frac{r^{k+2}s^{m-k}}
{(\om\whw)_{k+1}}\sum_{n=0}^\infty\frac{\widehat{g}(n+1)(n+1)}{n+k+1}\right)^2\right)=0.
\end{equation}
\end{proposition}

\begin{proof}
The boundedness can be verified by using \cite[Theorem~1.1]{CPPR}, an argument similar to that used in the proof of Theorem \ref{th:p=1}, and Fatou's lemma.

To verify the assertion on the compactness, assume first that $T_g:D^1_{\om}\rightarrow H^\infty$ is compact. A similar proof as that used in Proposition \ref{nonnegativepless1} together with Lemma \ref{le:duald1omhat} show that $\{G^{D^2_W}_{g,z}:z\in\D\}$ is relatively compact in $\BMOA'(\infty,\om)$. Hence, for a given $\varepsilon>0$, there exist $z_1,z_2,\ldots,z_N\in\D$ such that for each $z\in\D$, we have $\|G^{D^2_W}_{g,z}-G^{D^2_W}_{g,z_j}\|_{\BMOA'(\infty,\om)}<\varepsilon$ for some $j=j(z)\in\{1,\ldots,N\}$. Therefore, we get
	\begin{equation*}
	\begin{split}
	\left\|\left(G^{D^2_{\om\whw}}_{g,z}\right)_r\right\|_{\BMOA'(\infty,\om)}
	\le\e+\left\|\left(G^{D^2_{\om\whw}}_{g,z_j}\right)_r\right\|_{\BMOA'(\infty,\om)},
	\end{split}
	\end{equation*}
and thus
	$$
	\limsup_{r\rightarrow1^-}\sup_{z\in\D}\left\|(G^{D^2_{\om\whw}}_{g,z})_r\right\|_{\BMOA'(\infty,\om)}=0.
	$$
Consequently, by the hypothesis $\widehat{g}(n)\ge0$ for all $n\in\mathbb{N}\cup\{0\}$ and the first part of the proof, we have
\begin{align*}
\limsup_{r\rightarrow1^-}\limsup_{s\rightarrow1^-}\whw^2(r)\sum_{m=0}^\infty\frac{1-s}{(m+1)^2}\left(\sum_{k=0}^m\frac{r^{k+2}s^{m-k}}
{(\om\whw)_{k+1}}\sum_{n=0}^\infty\frac{\widehat{g}(n+1)(n+1)}{n+k+1}\right)^2=0.
\end{align*}

Conversely, assume that \eqref{eqcptp=1} holds. To prove the compactness of $T_g: D^1_\om\rightarrow H^\infty$, it suffices to show that each norm bounded family $\{f_n\}$ in $D^1_\om$ such that $f_n\rightarrow0$ uniformly on compact subsets of $\D$ satisfies $\|T_g(f_n)\|_{H^\infty}\rightarrow0$, as $n\to\infty$. Let now $\{f_n\}$ be such a family. By the uniform convergence we may choose $N=N(\varepsilon,R)\in\mathbb{N}$ such that $\max\{|f_n(0)|,|f'_n(0)|\}<\varepsilon$ for all $n\ge N$ and $\xi\in\overline{D(0,R)}$. It is easy to see that
\begin{equation}\label{eqp=11}
	\sup_{z\in\D}(|(G^{D^2_{\om\whw}}_{g,z})'(0)||(f_n)'(0)|
	+|G^{D^2_{\om\whw}}_{g,z}(0)||f_n(0)|)
	\lesssim\varepsilon \sup_{z\in\D}\|G^{D^2_{\om\whw}}_{g,z}\|_{\BMOA'(\infty,\om)}
	\lesssim\varepsilon,\quad n\ge N.
	\end{equation}	
Further, by \eqref{eqcptp=1}, for each $\varepsilon>0$ there exists an $R=R(\varepsilon)\in(0,1)$ such that
	$$
	\sup_{r\ge R, z\in\D\setminus D(0,R)}\whw(r)\left\|\left(G^{D^2_{\om\whw}}_{g,z}\right)'_r\right\|_{\BMOA}<\varepsilon.
	$$
Therefore, this together with \eqref{eqp=11} and the well-known duality $(H^1)^\star\simeq\BMOA$ via the $H^2$-pairing, implies
	\begin{equation*}
	\begin{split}
    \|T_g(f_n)\|_{H^\infty}
		&=\sup_{z\in\D\setminus D(0,R)}\left|\langle f_n, G^{D^2_{\om\whw}}_{g,z}\rangle_{D^2_{\om\whw}}\right|\\
		&\le\sup_{z\in\D\setminus D(0,R)}\left|\int_\D f'_n(\zeta)(G^{D^2_{\om\whw}}_{g,z})'(\zeta)\om(\zeta)\whw(\zeta)\,dA(\zeta)\right|
		+\left|G^{D^2_{\om\whw}}_{g,z}(0)\right||f_n(0)|\\
	&\lesssim\varepsilon+\sup_{z\in\D\setminus D(0,R)}\int_R^1\om(r)\whw(r)r \left|\int_0^{2\pi}f'_n(re^{i\theta})\left(G^{D^2_{\om\whw}}_{g,z}\right)'(re^{i\theta})\,d\theta\right|\,dr\\
    &\lesssim\varepsilon+\sup_{z\in\D\setminus D(0,R)}\int_{R}^1\om(r)\whw(r)r\|(G^{D^2_{\om\whw}}_{g,z})_r'\|_{\BMOA}\int_0^{2\pi}|f'_n(re^{i\theta})|\,d\theta\,dr
    \\ & \lesssim\varepsilon+ \sup_{r\ge R, z\in\D\setminus D(0,R)}\whw(r)\left\|\left(G^{D^2_{\om\whw}}_{g,z}\right)'_r\right\|_{\BMOA}
    \|f_n\|_{D^1_\omega}
		\lesssim\varepsilon,\quad n\ge N.
	\end{split}
	\end{equation*}
Hence $\lim_{n\rightarrow\infty}\|T_g(f_n)\|_{H^\infty}=0$. The proof is complete.
\end{proof}

\section{Case $1<p<\infty$}\label{Sec:p>1}

For each $g\in\H(\D)$, with Maclaurin series expansion $g(z)=\sum_{k=0}^\infty \widehat{g}(k)z^k$, consider the dyadic polynomials defined by $\Delta_0 g (z)=g(0)$ and $\Delta_n g (z)=\sum_{k=2^{n}}^{2^{n+1}-1}\widehat{g}(k)z^k$ for all $n\in\N$ and $z\in\D$. Then, obviously, $g=\sum_{n=0}^\infty\Delta_n g$. Further, write $\Delta_0=1$ and $\Delta_n(z)= \sum_{k=2^{n}}^{2^{n+1}-1}z^k$ for all $n\in\N$ and $z\in\D$. Then \cite[Lemma~2.7]{CPPR} shows that
	\begin{equation}\label{Deltanorm}
	\|\Delta_n\|_{H^p}\asymp 2^{\frac{n}{p'}},\quad 1<p<\infty,\quad n\in \N\cup\{0\}.
	\end{equation}
For $a\in\D$, denote $f_a(z)=f(az)$ for all $z\in\D$. Without loss of generality, one can suppose that $\int_0^1\omega(t)dt=1$. For each $n\in\N\cup\{0\}$, define $r_n=r_n(\omega)\in[0,1)$ by
$$
\widehat{\om}(r_n)=\int_{r_n}^1\om(s)ds=\frac{1}{2^n}.
$$
For each $x\in[0,\infty)$, let $E(x)\in\N\cup\{0\}$ such that $E(x)\le x<E(x)+1$, and set $M_n=E(\frac{1}{1-r_n})$. Write
$$
I_{\om}(0)=\{k\in\N\cup\{0\}:k<M_1\}
$$
and
	$$
	I_{\om}(n)=\{k\in\N:M_n\leq k<M_{n+1}\},\quad n\in\N.
	$$
If $f(z)=\sum_{n=0}^{\infty}\widehat{f}(n)z^n$ is analytic in $\D$, define the polynomials $\Delta_n^{\om}f$ by
	$$
	\Delta_n^{\om}f(z)=\sum_{k\in I_{\om}(n)}\widehat{f}(k)z^k,\quad n\in\N\cup\{0\}.
	$$
The Hadamard product of $f,g\in\H(\D)$ is formally defined by $(f*g)(z)=\sum_{n=0}^\infty\widehat{f}(n)\widehat{g}(n)z^n$.

We next prove the statement in Theorem~\ref{nonnegativepless1intro}(iii).

\begin{theorem}\label{TgDpom}
Let $1<p<\infty$ and $\omega\in\DD$, and let $g\in H^\infty$ with $\widehat{g}(n)\geq0$ for all $n\in\mathbb{N}\cup\{0\}$. Then $T_g: D^p_{\omega}\rightarrow H^\infty$ is bounded (equivalently compact) if and only if
	\begin{equation}\label{condition}
	\sum_{k=0}^{\infty}\frac{(k+1)^{-2}}{(\omega_k)^{p'-1}}
	\left(\sum_{n=0}^{\infty}\frac{(n+1)\widehat{g}(n+1)}{n+k+1}\right)^{p'}<\infty.
	\end{equation}
Moreover,
	\begin{equation}\label{conditioniuy}
	\|T_g\|_{D^p_{\om}\rightarrow H^{\infty}}^{p'}\asymp\sum_{k=0}^{\infty}\frac{(k+1)^{-2}}{(\omega_k)^{p'-1}}
	\left(\sum_{n=0}^{\infty}\frac{(n+1)\widehat{g}(n+1)}{n+k+1}\right)^{p'}.
	\end{equation}
\end{theorem}

\begin{proof}
We begin with showing that \eqref{condition} is a sufficient condition for $T_g: D^p_{\omega}\rightarrow H^\infty$ to be bounded. By \cite[Theorem~1.1]{CPPR}, Lemma~\ref{le:dualdp}, \cite[Theorem~3.4]{PR2017} and \cite[Lemma~E]{PeR}, we deduce
	\begin{equation}\label{Eq:bd1}
	\begin{split}
	&\|T_g\|^{p'}_{D^p_{\om}\rightarrow H^\infty}\asymp\sup_{z\in\D}\|G^{D^2_{\om}}_{g,z}\|^{p'}_{D^{p'}_{\om}}\asymp
	\sup_{z\in\D}\sum_{j=0}^{\infty}2^{-j}\|\Delta_j^{\om}(G^{D^2_{\om}}_{g,z})^{\prime}\|^{p'}_{H^{p'}}\\
	&\asymp\sup_{z\in\D}\sum_{j=0}^{\infty}{2^{-j}}\left\|\sum_{k=M_j}^{M_{j+1}-1}\frac{1}{(\om_{2k+1})(k+1)}
	\left(\sum_{n=0}^{\infty}\frac{(n+1)\widehat{g}(n+1)}{n+k+2}\overline{z}^{n+k+2}\right)w^k\right\|^{p'}_{H^{p'}}\\
	&\asymp\sup_{z\in\D}\sum_{j=0}^{\infty}\frac{1}{(\om_{2M_j+1})^{p'-1}}\left\|\sum_{k=M_j}^{M_{j+1}-1}\frac{1}{k+1}
	\left(\sum_{n=0}^{\infty}\frac{(n+1)\widehat{g}(n+1)}{n+k+1}\overline{z}^{n+k+1}\right)w^k\right\|^{p'}_{H^{p'}},
	\end{split}
\end{equation}
where the last step is valid because $\om_{2x}\asymp\om_x$, $1\le x<\infty$, and $\om_{M_{j+1}}\asymp\om_{M_j}\asymp2^{-j}$, $j\in\N$, by Lemma~\ref{Eqdhat}(iii). We consider two different cases. Let first $1<p\le2$. Then $2\le p'<\infty$, and hence \eqref{Hardy2} yields
	\begin{equation*}
	\begin{split}
	\|T_g\|^{p'}_{D^p_{\om}\rightarrow H^\infty}
	&\asymp\sup_{z\in\D}\sum_{j=0}^{\infty}\frac{1}{(\om_{2M_j+1})^{p'-1}}\left\|\sum_{k=M_j}^{M_{j+1}-1}
	\frac{1}{k+1}\left(\sum_{n=0}^{\infty}\frac{(n+1)\widehat{g}(n+1)}{n+k+1}\overline{z}^{n+k+1}\right)w^k\right\|^{p'}_{H^{p'}}\\
	&\lesssim \sum_{j=0}^{\infty}\frac{1}{(\om_{2M_j+1})^{p'-1}}\sum_{k=M_j}^{M_{j+1}-1}(k+1)^{-2}
	\left(\sum_{n=0}^{\infty}\frac{(n+1)|\widehat{g}(n+1)|}{n+k+1}\right)^{p'}\\
	&\asymp\sum_{k=0}^{\infty}\frac{(k+1)^{-2}}{(\omega_{k})^{p'-1}}
	\left(\sum_{n=0}^{\infty}\frac{(n+1)|\widehat{g}(n+1)|}{n+k+1}\right)^{p'}.
	\end{split}
	\end{equation*}

If $2<p<\infty$, then \cite[Theorem~1.1]{CPPR}, Lemma~\ref{le:dualdp}, \eqref{Hardy1}, \cite[Theorem~2.1]{MatPav}, \cite[Lemma~E]{PeR}, Fubini's theorem and Lemma~\ref{Eqdhat}(iii) imply
	\begin{equation*}
	\begin{split}
	\|T_g\|^{p'}_{D^p_{\om}\rightarrow H^\infty}
	&
	\asymp\sup_{z\in\D}\int_0^1\|(G^{D^2_{\om}}_{g,z})^{\prime}_{r}\|_{H^{p'}}^{p'}\om(r)r\,dr
	\lesssim\sup_{z\in\D}\int_0^1\|(G^{D^2_{\om}}_{g,z})^{\prime}_{r}\|_{D^{p'}_{p'-1}}^{p'}\om(r)rdr\\
	&\asymp\sup_{z\in\D}\int_0^1\sum_{j=0}^{\infty}\|\Delta_j*(G^{D^2_{\om}}_{g,z})^{\prime}_{r}\|_{H^{p'}}^{p'}\om(r)dr\\
	&\asymp\sup_{z\in\D}\int_0^1\sum_{j=0}^{\infty}\left\|\sum_{k=2^{j}}^{2^{j+1}-1}\frac{r^k}{2(k+1)\om_{2k+1}}
	\left(\sum_{n=0}^{\infty}\frac{\overline{\widehat{g}(n+1)}(n+1)\overline{z}^{n+k+2}}{n+k+2}
	\right)\zeta^k\right\|^{p'}_{H^{p'}}\om(r)dr\\
	&\lesssim\sup_{z\in\D}\int_0^1\sum_{j=0}^{\infty}\frac{r^{2^j}}{2^{(j+1)p'}(\om_{2^{j+1}})^{p'}}
	\left\|\sum_{k=2^{j}}^{2^{j+1}-1}
	\left(\sum_{n=0}^{\infty}\frac{\overline{\widehat{g}(n+1)}(n+1)\overline{z}^{n+k+2}}{n+k+2}
	\right)\zeta^k\right\|^{p'}_{H^{p'}}\om(r)dr\\
	&\asymp\sup_{z\in\D}\sum_{j=0}^{\infty}\frac{1}{2^{jp'}(\om_{2^j})^{p'-1}}
	\left\|\sum_{k=2^{j}}^{2^{j+1}-1}
	\left(\sum_{n=0}^{\infty}\frac{\overline{\widehat{g}(n+1)}(n+1)\overline{z}^{n+k+2}}{n+k+2}
	\right)\zeta^k\right\|^{p'}_{H^{p'}}.
	\end{split}
	\end{equation*}
Now \cite[Proposition~9]{PRW} shows that
	$$
	\left\|\sum_{k=2^{j}}^{2^{j+1}-1}
	\left(\sum_{n=0}^{\infty}\frac{\overline{\widehat{g}(n+1)}(n+1)\overline{z}^{n+k+2}}{n+k+2}
	\right)\zeta^k\right\|^{p'}_{H^{p'}}\lesssim\|\Delta_j\|_{H^{p'}}^{p'}
	\left(\sum_{n=0}^{\infty}\frac{(n+1)|\widehat{g}(n+1)||z|^{n+2}}{n+2^{j-1}+2}\right)^{p'},
	$$
and hence \eqref{Deltanorm} yields
	\begin{align*}
	\|T_g\|^{p'}_{D^p_{\om}\rightarrow H^\infty}
	&\lesssim \sup_{z\in\D}\sum_{j=0}^{\infty}\frac{1}{2^{jp'}(\om_{2^j})^{p'-1}}\|\Delta_j\|_{H^{p'}}^{p'}
	\left(\sum_{n=0}^{\infty}\frac{(n+1)|\widehat{g}(n+1)||z|^{n+2}}{n+2^{j-1}+2}\right)^{p'}\\
	&\asymp\sum_{j=0}^{\infty}\frac{2^{j(p'-1)}}{2^{jp'}(\om_{2^j})^{p'-1}}
	\left(\sum_{n=0}^{\infty}\frac{(n+1)|\widehat{g}(n+1)|}{n+2^{j-1}+2}\right)^{p'}\\
	&\asymp\sum_{k=0}^{\infty}\frac{(k+1)^{-2}}{(\om_{k})^{p'-1}}
	\left(\sum_{n=0}^{\infty}\frac{(n+1)|\widehat{g}(n+1)|}{n+k+1}\right)^{p'}.
	\end{align*}
Thus \eqref{condition} is a sufficient condition for $T_g: D^p_{\omega}\rightarrow H^\infty$ to be bounded.

Conversely, assume $g\in\H(\D)$ such that $\widehat{g}(n)\geq0$ for all $n\in\mathbb{N}\cup\{0\}$, and $T_g:D^p_{\omega}\rightarrow H^{\infty}$ is bounded. Then \eqref{Eq:bd1} implies
	\begin{equation*}
	\begin{split}
	\|T_g\|^{p'}_{D^p_{\om}\rightarrow H^\infty}
	\gtrsim\sup_{0\leq x<1}\sum_{j=0}^{\infty}\frac{1}{(\om_{M_j})^{p'-1}}\left\|\sum_{k=M_j}^{M_{j+1}-1}
	\frac{1}{k+1}\left(\sum_{n=0}^{\infty}\frac{(n+1)\widehat{g}(n+1)}{n+k+1}{x}^{n+k+1}\right)w^k\right\|^{p'}_{H^{p'}}.
	\end{split}
	\end{equation*}
For each $x\in(0,1)$, the coefficients
	$$
	\widehat{G^{H^2}_{g,x}}(k)=\frac{1}{k+1}\sum_{n=0}^\infty \frac{(n+1)\widehat{g}(n+1)x^{n+k+1}}{n+k+1}, \quad k\in\N\cup\{0\},
	$$
form a sequence of non-negative and decreasing numbers. Therefore \eqref{Eq:old-case} implies
	$$
	\|T_g\|_{D^p_{\om}\rightarrow H^{\infty}}^{p'}\gtrsim\sum_{k=0}^{\infty}\frac{(k+1)^{-2}}{(\omega_{k})^{p'-1}}
	\left(\sum_{n=0}^{\infty}\frac{(n+1)\widehat{g}(n+1)}{n+k+1}\right)^{p'}.
	$$
Thus \eqref{condition} is satisfied. Further, the norm estimate \eqref{conditioniuy} follows from the proof above.

To complete the proof, it suffices to show that \eqref{condition} yields the compactness of $T_g: D^p_{\omega}\rightarrow H^\infty$.
We claim that
	\begin{equation}\label{compactness}
	L=\lim_{R\rightarrow1^-}\sup_{z\in\D}\int_{\D\setminus\overline{\D(0,R)}}
	|(G^{D^2_{\om}}_{g,z})^{\prime}(\zeta)|^{p'}\om(\zeta)\,dA(\zeta)=0.
	\end{equation}

If $2\le p'<\infty$, then \eqref{Hardy2}, Fubini's theorem and Lemma~\ref{Eqdhat}(iii) yield
	\begin{align*}
	L&\asymp\lim_{R\rightarrow1^-}\sup_{z\in\D}\int_R^1\|(G^{D^2_{\om}}_{g,z})^{\prime}_{r}\|_{H^{p'}}^{p'}\om(r)rdr\\
	&\lesssim\lim_{R\rightarrow1^-}\int_R^1\sum_{k=0}^{\infty}\frac{(k+1)^{-2}}{(\om_k)^{p'}}
	\left(\sum_{n=0}^{\infty}\frac{(n+1)\widehat{g}(n+1)}{n+k+1}\right)^{p'}r^{kp'}\om(r)dr\\
	&=\sum_{k=0}^{\infty}\frac{(k+1)^{-2}}{(\om_k)^{p'-1}}
	\left(\sum_{n=0}^{\infty}\frac{(n+1)\widehat{g}(n+1)}{n+k+1}\right)^{p'}\frac{1}{\om_k}\int_{R}^1r^{kp'}\om(r)dr.
	\end{align*}
Since $\frac{1}{\om_k}\int_{R}^1r^{kp'}\om(r)dr\lesssim1$ for all $k\in\N\cup\{0\}$ and $0\le R<1$, the dominated convergence theorem implies \eqref{compactness}.

If $1<p'<2$, then by \eqref{Hardy1}, together with an argument similar to the case $2<p<\infty$ in the first part of the proof and the dominated convergence theorem, we deduce \eqref{compactness}.

Let now $\{f_n\}$ be a norm bounded family of functions in $D^p_{\om}$ such that $f_n\rightarrow0$ uniformly on compact subsets of $\D$. By \eqref{compactness}, for each $\varepsilon>0$, there exists $R=R(\varepsilon)\in(0,1)$ such that
	$$
	\sup_{z\in\D}\int_{\D\setminus\overline{\D(0,R)}}
	|(G^{D^2_{\om}}_{g,z})^{\prime}(\zeta)|^{p'}\om(\zeta)\,dA(\zeta)<\varepsilon^{p'}.
	$$
Further, by the uniform convergence we may choose $N=N(\varepsilon,R)\in\N$ such that $\max\{|f_n(0)|,|f_n'(\zeta)|\}<\varepsilon$ for all $n\ge N$ and $\zeta\in\overline\D(0,R)$. Therefore \cite[(2.4)]{CPPR} and H\"{o}lder's inequality yield
	\begin{align*}
	\|T_gf_n\|_{H^{\infty}}&=\sup_{z\in\D}|\langle f_n, G^{D^2_{\om}}_{g,z}\rangle_{D^2_{\om}}|\\
	&\lesssim\sup_{z\in\D}\int_{\D}|f_n'(\zeta)||(G^{D^2_{\om}}_{g,z})^{\prime}(\zeta)|\om(\zeta)\,dA(\zeta)+|f_n(0)||g(z)-g(0)|\\
	&\lesssim\sup_{z\in\D}\left(\left(\int_{\overline{\D(0,R)}}+\int_{\D\backslash\overline{\D(0,R)}}\right)
	|f_n'(\zeta)||(G^{D^2_{\om}}_{g,z})'(\zeta)|\om(\zeta) dA(\zeta)\right)\\
	&\quad+|f_n(0)| \|g\|_{H^\infty}\\
	&\lesssim\varepsilon\left(\|(G^{D^2_{\om}}_{g,z})^{\prime}\|^{p'}_{D^{p'}_{\om}}+\sup_{n}\|f_n\|_{D^p_{\om}}+\|g\|_{H^\infty}\right)
	\lesssim\e,\quad n\ge N.
	\end{align*}
Hence $T_g:D^p_{\om}\to H^\infty$ is compact by \cite[Lemma~3.6]{Tjani}.
\end{proof}

The next result says that the spaces $\HL^{\om}_{p}$, $A^p_{\om}$ and $D^p_{\widetilde{\om}_{[p]}}$ play the same role when $T_g$ acts from one of them to $H^\infty$ in terms of boundedness and compactness, provided $g$ has non-negative Maclaurin series coefficient and $\om\in\DD$.

\begin{proposition}\label{TgHLpom}
Let $1<p<\infty$, $\omega\in\DD$, $g\in H^\infty$ with $\widehat{g}(n)\geq0$ for all $n\in\mathbb{N}\cup\{0\}$, and let $X^{\om}_{p}\in\{\HL^{\om}_{p}, A^p_{\om}, D^p_{\widetilde{\om}_{[p]}}\}$. Then $T_g:X^{\om}_p\rightarrow H^\infty$ is bounded (equivalently compact) if and only if
	\begin{equation}\label{conditionApom}
	\sum_{k=0}^{\infty}\frac{(k+1)^{p'-2}}{(\omega_{k})^{p'-1}}
	\left(\sum_{n=0}^{\infty}\frac{(n+1)\widehat{g}(n+1)}{n+k+1}\right)^{p'}<\infty.
	\end{equation}
Moreover,
	\begin{equation}\label{conditionApompoiu}
	\|T_g\|_{X_{p}^{\om}\rightarrow H^{\infty}}^{p'}\asymp\sum_{k=0}^{\infty}\frac{(k+1)^{p'-2}}{(\omega_{k})^{p'-1}}
	\left(\sum_{n=0}^{\infty}\frac{(n+1)\widehat{g}(n+1)}{n+k+1}\right)^{p'}.
	\end{equation}
\end{proposition}

\begin{proof}
Assume first that $X^\om_p=D^p_{\widetilde{\om}_{[p]}}$. By applying Lemma~\ref{Eqdhat}(iii) we deduce $(\widetilde{\om}_{[p]})_{k}\asymp\om_{k}(k+1)^{-p}$ for all $k\in\N\cup\{0\}$.
Therefore Theorem~\ref{TgDpom}, applied to $\widetilde{\om}_{[p]}$, yields
	$$
	\|T_g\|_{D^p_{\widetilde{\om}_{[p]}}\rightarrow H^{\infty}}^{p'}
	\asymp\sum_{k=0}^{\infty}\frac{(k+1)^{p'-2}}{(\omega_{k})^{p'-1}}
	\left(\sum_{n=0}^{\infty}\frac{(n+1)\widehat{g}(n+1)}{n+k+1}\right)^{p'}.
	$$
Moreover, the boundedness of $T_g: D^p_{\widetilde{\om}_{[p]}}\rightarrow H^\infty$ is equivalent to its compactness. Therefore, by Proposition~\ref{inclusion}, to finish the proof, it suffices to prove the statement for $X^\om_p=\HL^\om_p$.

By \cite[Theorem 2.2]{CPPR} and Lemma~\ref{le:dualHLp}, $T_g:\HL_p^{\om}\to H^{\infty}$ is bounded if and only if $\sup_{z\in\D}\|G^{A^2_\om}_{g,z}\|_{\HL^{\om}_{p'}}<\infty$, and further, $\|T_g\|_{\HL_{p}^{\om}\rightarrow H^{\infty}}\asymp\sup_{z\in\D}\|G^{A^2_\om}_{g,z}\|_{\HL^{\om}_{p'}}$. Obviously,
	\begin{align*}
	\sup_{z\in\D}\|G^{A^2_\om}_{g,z}\|_{\HL^{\om}_{p'}}^{p'}
	&\asymp\sup_{z\in\D}\sum_{k=0}^{\infty}\frac{(k+1)^{p'-2}}{(\om_{kp+1})^{p'-1}}\left|\sum_{n=0}^{\infty}
	\frac{(n+1)\widehat{g}(n+1)\overline{z}^{n+k+1}}{n+k+1}\right|^{p'}\\
	&\lesssim\sum_{k=0}^{\infty}\frac{(k+1)^{p'-2}}{(\omega_{k})^{p'-1}}
	\left(\sum_{n=0}^{\infty}\frac{(n+1)\widehat{g}(n+1)}{n+k+1}\right)^{p'},
	\end{align*}
and by Fatou's lemma,
	$$
	\sup_{z\in\D}\|G^{A^2_\om}_{g,z}\|_{\HL^{\om}_{p'}}^{p'}\gtrsim
	\sup_{0\leq x<1}\|G^{A^2_\om}_{g,x}\|_{\HL^{\om}_{p'}}^{p'}\asymp
	\sum_{k=0}^{\infty}\frac{(k+1)^{p'-2}}{(\omega_{k})^{p'-1}}
	\left(\sum_{n=0}^{\infty}\frac{(n+1)\widehat{g}(n+1)}{n+k+1}\right)^{p'}.
	$$
Finally, we will prove that \eqref{conditionApom} implies the compactness of $T_g:\HL_p^{\om}\to H^{\infty}$. To this end, let $\{f_n\}$ be a norm bounded family in $\HL^\om_{p}$ such that $f_n\rightarrow0$ uniformly on compact subsets of $\D$. Then, for each $\varepsilon>0$ there exists $k_0=k_0(\varepsilon)\in\mathbb{N}$ such that
	$$
	\sum_{k=k_0}^{\infty}\frac{(k+1)^{p'-2}}{\om_k^{p'-1}}\left(
	\sum_{n=0}^{\infty}\frac{(n+1)\widehat{g}(n+1)}{n+k+1}\right)^{p'}<\varepsilon^{p'}.
	$$
Furthermore, by the uniformly convergence we may choose an $n_0=n_0(\varepsilon)\in\mathbb{N}$ such that
$$
\sup_{n\geq n_0}\sum_{k=0}^{k_0-1}(k+1)^{p-2}|\widehat{f_n}(k)|^p\om_k<\varepsilon^p.
$$
Then H\"{o}lder's inequality yields
	\begin{align*}
	\|T_g(f_n)\|_{H^\infty}
	&=\sup_{z\in\D}|\langle f_n, G^{A^2_\om}_{g,z}\rangle_{A^2_\om}|\leq
	\sum_{k=0}^{\infty}|\widehat{f_n}(k)|\left(\sum_{n=0}^{\infty}\frac{(n+1)\widehat{g}(n+1)}{n+k+1}\right)\\
	&=\left(\sum_{k=0}^{k_0-1}+\sum_{k=k_0}^\infty\right)|\widehat{f_n}(k)|(k+1)^{\frac{p-2}{p}}\om_k^{\frac{1}{p}}\frac{(k+1)^{\frac{p'-2}{p'}}}{\om_k^{\frac{1}{p}}}
\left(\sum_{n=0}^{\infty}\frac{(n+1)\widehat{g}(n+1)}{n+k+1}\right)\\
	&\leq\left(\sum_{k=0}^{k_0-1}|\widehat{f_n}(k)|^p(n+1)^{p-2}\om_k\right)^{\frac1p}
\left(\sum_{k=0}^{k_0-1}\frac{(k+1)^{p'-2}}{\om_k^{p'-1}}\left(\sum_{n=0}^{\infty}
\frac{(n+1)\widehat{g}(n+1)}{n+k+1}\right)^{p'}\right)^{\frac{1}{p'}}\\
	&\quad+\left(\sum_{k=k_0}^{\infty}|\widehat{f_n}(k)|^p(n+1)^{p-2}\om_k\right)^{\frac1p}
\left(\sum_{k=k_0}^{\infty}\frac{(k+1)^{p'-2}}{\om_k^{p'-1}}\left(\sum_{n=0}^{\infty}
\frac{(n+1)\widehat{g}(n+1)}{n+k+1}\right)^{p'}\right)^{\frac{1}{p'}}\\
	&\lesssim\varepsilon,\quad n\geq n_0,
	\end{align*}
Hence $\lim_{n\rightarrow\infty}\|T_g(f_n)\|_{H^\infty}=0$. Consequently, $T_g:\HL^\om_p\rightarrow H^\infty$ is compact by \cite[Lemma~3.6]{Tjani}. The proof is complete.
\end{proof}

The proof of Theorem~\ref{compactDpless1} implies that, for each $0<p<\infty$, the boundedness of $I:D^p_{\om}\rightarrow D^1_0$ implies $g\in T(D^p_{\om},H^\infty)$ whenever $g'\in H^\infty$. In particular, each monomial satisfies $m_n\in T(D^p_{\om},H^{\infty})$ and thus $T(D^p_{\om},H^\infty)$ is not trivial. Conversely, if $0<p\le1$, then Theorems~\ref{D^pless1} and \ref{th:p=1} show that $T(D^p_{\om},H^\infty)$ being non-trivial implies the boundedness of $I:D^p_{\om}\rightarrow D^1_0$. This is no longer true if $1<p<\infty$. Indeed, by the Carleson embedding theorem \cite[Theorem~1]{PR2015}, for $1<p<\infty$, $I:D^p_{\om}\rightarrow D^1_0$ is bounded if and only the function
	$$
	B(z)=\int_{\Gamma(z)}\frac{dA(\zeta)}{\om(T(\zeta))}, \quad z\in\D\setminus\{0\},
	$$
belongs to $L^{p'}_{\om}$. Here
	$$
	\Gamma(z)=\left\{\zeta\in\D:|\arg z-\arg\zeta|<\frac12\left(1-\frac{|\zeta|}{|z|}\right)\right\}
	$$
is the lens-type region with vertex at $z\in\D\setminus\{0\}$ and $T(\zeta)=\{z\in\D:\zeta\in\Gamma(z)\}$. By Lemma~\ref{Eqdhat}(ii), we deduce
	\begin{align*}
	\|B\|^{p'}_{L^{p'}_{\om}}
	\asymp\int_0^1\left(\int_0^r\frac{r-s}{(1-s)\widehat{\om}(s)}ds\right)^{p'}\om(r)\,dr
	\gtrsim \int_0^1\frac{(1-r)^{p'}}{\whw(r)^{p'-1}}\frac{\om(r)}{\whw(r)}dr.
	\end{align*}
The standard radial weight $\om(z)=(1-|z|)^{\alpha}$, with $\frac{1}{p'-1}\le\alpha<\frac{2}{p'-1}$, satisfies
	$$
	\int_0^1\frac{(1-r)^{p'}}{\widehat{\om}(r)^{p'-1}}dr<\infty
	=\int_0^1\frac{(1-r)^{p'}}{\whw(r)^{p'-1}}\frac{\om(r)}{\whw(r)}dr
	\lesssim\|B\|^{p'}_{L^{p'}_{\om}}.
	$$
Thus the next result shows that $T(D^p_{\om},H^\infty)$ being non-trivial does not force $I:D^p_{\om}\rightarrow D^1_0$ to be bounded when $1<p<\infty$.

Theorem~\ref{intro:1p<infinity} follows from the next result.

\begin{theorem}\label{Dpconstant}
Let $1<p<\infty$, $\om\in\DD$ and $g\in H^{\infty}$. Then the following statements are equivalent:
\begin{itemize}
\item[(i)] $T(D^p_{\om}, H^{\infty})$ (equivalently $T_c(D^p_{\om}, H^{\infty})$) consists of constant functions only;
\item[(ii)] $\displaystyle \int_0^1\frac{(1-r)^{p'}}{\widehat{\om}(r)^{p'-1}}\,dr=\infty$;
\item[(iii)] $\displaystyle \sum_{k=0}^{\infty}\frac{1}{(k+1)^{2+p'}\om_k^{p'-1}}=\infty$;
\item[(iv)] $I:D^{p^*}_{\whw_{[x]}}\rightarrow D^{\frac{p^*}{p}}_y$ is unbounded, where $x=p(y-1)$, $x>y>p'$ and $0<p^*<\infty$.
\end{itemize}
\end{theorem}

\begin{proof}
By Lemma~\ref{Eqdhat}(iii), the moments of $\om$ satisfy $\om_x\asymp\widehat{\om}(1-\frac{1}{x+1})$ for all $x\in[0,\infty)$, and hence
	\begin{equation*}
	\begin{split}
	\sum_{k=0}^{\infty}\frac{1}{(k+1)^{2+p'}\om_k^{p'-1}}
	&\asymp\int_0^\infty\frac{dx}{(x+1)^{2+p'}\om_x^{p'-1}}
	\asymp\int_0^\infty\frac{dx}{(x+1)^{2+p'}\widehat{\om}\left(1-\frac1{x+1}\right)^{p'-1}}\\
	&=\int_0^1\frac{(1-r)^{p'}}{\widehat{\om}(r)^{p'-1}}\,dr.
	\end{split}
	\end{equation*}
Thus (ii) and (iii) are equivalent.

If (iii) does not hold, then Theorem~\ref{TgDpom} shows that the identity mapping $z\mapsto z$ belongs to $T_c(D^p_{\om},H^\infty)$. Therefore (i) implies (iii).

We next show that (iii) implies (i). Assume on the contrary to (i) that there exists a non-constant $g\in T(D^p_{\om}, H^{\infty})$. Then $\widehat{g}(N+1)\ne0$ for some $N\in\N\cup\{0\}$. We consider two cases. If $1<p'\le2$, then \cite[Theorem 1.1]{CPPR}, Proposition~\ref{inclusion}(i), and \eqref{Hardy1} yield
	\begin{align*}
	\infty&>\|T_g\|^{p'}_{D^p_\om\rightarrow H^\infty}
	\asymp\sup_{z\in\D}\|G^{D^2_\om}_{g,z}\|^{p'}_{D^{p'}_{\om}}\gtrsim
	\sup_{z\in\D}\|(G^{D^2_\om}_{g,z})^{\prime}\|^{p'}_{\HL_{p'}^\om}\\
	&\asymp\sup_{z\in\D}\sum_{k=0}^{\infty}\frac{(k+1)^{-2}}{\om_k^{p'-1}}\left|\sum_{n=0}^{\infty}
	\frac{\overline{\widehat{g}(n+1)}(n+1)\overline{z}^{n+k+2}}{n+k+2}\right|^{p'}\\
	&=\sup_{0\leq r<1}\sup_{\theta\in[0,2\pi)}\sum_{k=0}^{\infty}\frac{(k+1)^{-2}r^{(k+2)p'}}
	{\om_k^{p'-1}}\left|\sum_{n=0}^{\infty}\frac{{\widehat{g}(n+1)}(n+1){r}^{n}e^{in\theta}}{n+k+2}\right|^{p'}\\
	&\geq\sup_{0\leq r<1}\frac{1}{2\pi}\sum_{k=0}^{\infty}\frac{(k+1)^{-2}r^{(k+2)p'}}{\om_k^{p'-1}}
	\int_0^{2\pi}\left|\sum_{n=0}^{\infty}\frac{{\widehat{g}(n+1)}(n+1){r}^{n}e^{in\theta}}{n+k+2}\right|^{p'}d\theta\\
	&\gtrsim\sup_{0\leq r<1}\sum_{k=0}^{\infty}\frac{(k+1)^{-2}r^{(k+2)p'}}{\om_k^{p'-1}}
	\left(\frac{|\widehat{g}(N+1)|(N+1){r}^{N}}{N+k+2}\right)^{p'}(N+1)^{p'-2}\\
	&\asymp\sum_{k=0}^{\infty}\frac{1}{(k+1)^{2+p'}\om_k^{p'-1}}.
	\end{align*}
Observe that, in this case $N$ is irrelevant but by and large, using the boundedness of the M. Riesz projection in the last inequality we can get a better estimate. Anyhow, this contradicts (iii), and thus (i) is satisfied.

Let now $2\le p'<\infty$. Then, by \eqref{Hardy2}, \cite[Theorem 2.1]{MatPav}, two consecutive applications of \cite[Lemma~E]{PeR}, Lemma~\ref{Eqdhat}(iv) and Fubini's theorem, we deduce
	\begin{equation*}
	\begin{split}
	\infty&>\|T_g\|^{p'}_{D^p_{\om}\rightarrow H^\infty}\asymp\sup_{z\in\D}\|G^{D^2_{\om}}_{g,z}\|^{p'}_{D^{p'}_{\om}}\asymp
	\sup_{z\in\D}\int_0^1\|(G^{D^2_{\om}}_{g,z})^{\prime}_{r}\|_{H^{p'}}^{p'}\om(r)rdr\\
	&\gtrsim\sup_{z\in\D}\int_0^1\|(G^{D^2_{\om}}_{g,z})^{\prime}_{r}\|_{D^{p'}_{p'-1}}^{p'}\om(r)rdr\\
	&\asymp\sup_{z\in\D}\int_0^1\sum_{j=0}^{\infty}\|\Delta_j*(G^{D^2_{\om}}_{g,z})^{\prime}_{r}\|_{H^{p'}}^{p'}\om(r)dr\\
	&\asymp\sup_{z\in\D}\int_0^1\sum_{j=0}^{\infty}\left\|\sum_{k=2^{j}}^{2^{j+1}-1}\frac{r^k}{2(k+1)\om_{2k+1}}
	\left(\sum_{n=0}^{\infty}\frac{\overline{\widehat{g}(n+1)}(n+1)\overline{z}^{n+2}}{n+k+2}
	\right)(|z|\zeta)^k\right\|^{p'}_{H^{p'}}\om(r)dr\\
	&\gtrsim\sup_{z\in\D}\int_0^1\sum_{j=0}^{\infty}\frac{(r|z|)^{2^{j+1}}}{2^{(j+1)p'}(\om_{2^{j+1}})^{p'}}
	\left\|\sum_{k=2^{j}}^{2^{j+1}-1}
	\left(\sum_{n=0}^{\infty}\frac{\overline{\widehat{g}(n+1)}(n+1)\overline{z}^{n+2}}{n+k+2}
	\right)\zeta^k\right\|^{p'}_{H^{p'}}\om(r)dr\\
	&\asymp\sup_{z\in\D}\sum_{j=0}^{\infty}\frac{|z|^{2^{j+1}}}{2^{(j+1)p'}(\om_{2^{j+1}})^{p'-1}}
	\left\|\sum_{k=2^{j}}^{2^{j+1}-1}
	\left(\sum_{n=0}^{\infty}\frac{\overline{\widehat{g}(n+1)}(n+1)\overline{z}^{n+2}}{n+k+2}
	\right)\zeta^k\right\|^{p'}_{H^{p'}}
	\end{split}
	\end{equation*}
from which Fubini's theorem and M.~Riesz projection theorem yield
	\begin{align*}
	\infty&>\sup_{0\le s<1}\sum_{j=0}^{\infty}\frac{s^{2^{j+1}}}{2^{(j+1)p'}(\om_{2^{j+1}})^{p'-1}}
	\int_0^{2\pi}\int_0^{2\pi}\left|\sum_{k=2^{j}}^{2^{j+1}-1}
	\left(\sum_{n=0}^{\infty}\frac{\overline{\widehat{g}(n+1)}(n+1)(se^{it})^{n+2}}{n+k+2}
	\right)e^{ik\theta}\right|^{p'}dt~d\theta\\
	&=\sup_{0\leq s<1}\sum_{j=0}^{\infty}\frac{s^{2^{j+1}}}{2^{(j+1)p'}(\om_{2^{j+1}})^{p'-1}}
	\int_0^{2\pi}\int_0^{2\pi}\left|\sum_{n=0}^{\infty}\overline{\widehat{g}(n+1)}(n+1)
	\left(\sum_{k=2^{j}}^{2^{j+1}-1}\frac{e^{ik\theta}}{n+k+2}
	\right)(se^{it})^{n+2}\right|^{p'}d\theta~ dt\\
	&\gtrsim\sup_{0\leq s<1}\sum_{j=0}^{\infty}\frac{s^{2^{j+1}}}{2^{(j+1)p'}(\om_{2^{j+1}})^{p'-1}}
	\int_0^{2\pi}\int_0^{2\pi}\left|\overline{\widehat{g}(N+1)}(N+1)
	\left(\sum_{k=2^{j}}^{2^{j+1}-1}\frac{e^{ik\theta}}{N+k+2}
	\right)(se^{it})^{N+2}\right|^{p'}d\theta~ dt\\
	&\asymp\sum_{j=0}^{\infty}\frac{1}{2^{(j+1)p'}(\om_{2^{j+1}})^{p'-1}}
	\int_0^{2\pi}\left|\sum_{k=2^{j}}^{2^{j+1}-1}\frac{e^{ik\theta}}{N+k+2}\right|^{p'}d\theta.
	\end{align*}
Now, by \cite[Lemma~E]{PeR} and \cite[Lemma~2.7]{CPPR}, with $M_n=2^n$ and $M_{n+1}=2^{n+1}$, we deduce
	$$
	\int_0^{2\pi}
	\left|\sum_{k=2^{j}}^{2^{j+1}-1}\frac{e^{ik\theta}}{N+k+2}\right|^{p'}d\theta\asymp 2^{-(j+1)},\quad j\in\N.
	$$
Consequently, since $\om\in\DD$, we finally obtain
	$$
	\infty>\|T_g\|^{p'}_{D^p_{\om}}
	\gtrsim\sum_{j=0}^{\infty}\frac{1}{2^{(j+1)(p'+1)}(\om_{2^{j+1}})^{p'-1}}
	\asymp\sum_{k=0}^{\infty}\frac{1}{(k+1)^{p'+2}\om_k^{p'-1}},
	$$
which is contradicts (iii), and thus (i) is satisfied.

We complete the proof by showing that (ii) and (iv) are equivalent. By \cite[Theorem~1]{PR2015}, $I:D^{p^*}_{\whw_{[x]}}\rightarrow D^{\frac{p^*}{p}}_y$ is bounded if and only if
	\begin{equation*}
	\begin{split}
	\infty
	&>\int_{\D}\whw(z)(1-|z|)^x\left(\int_{\Gamma(z)}\frac{(1-|\zeta|)^y}{\whw_{[x]}(S(\zeta))}\,dA(\zeta)\right)^{p'}dA(z)\\
	&\asymp\int_0^1\whw(r)(1-r)^x\left(\int_0^r\frac{(r-s)(1-s)^y}{\whw(s)(1-s)^{x+2}}s\,ds\right)^{p'}dr
	=I_\om.
	\end{split}
	\end{equation*}
On one hand, by Lemma~\ref{Eqdhat}(ii), there exists a constant $\beta=\beta(\om)>0$ such that
	\begin{align*}
	I_\om&\gtrsim\int^1_0\frac{(1-r)^{x+\beta p'}}{\whw(r)^{p'-1}}\left(\int_0^r\frac{r-s}{(1-s)^{\beta+x+2-y}}s\,ds\right)^{p'}dr
	\asymp\int_0^1\frac{(1-r)^{x+p'y-xp'}}{\whw(r)^{p'-1}}dr.
	\end{align*}
On the other hand,
	\begin{align*}
	I_\om
	&\lesssim\int_0^1\frac{(1-r)^x}{\whw(r)^{p'-1}}\left(\int_0^r\frac{r-s}{(1-s)^{x+2-y}}s\,ds\right)^{p'}dr
	\asymp\int_0^1\frac{(1-r)^{x+p'y-xp'}}{\whw(r)^{p'-1}}dr.
\end{align*}
Since $x=p(y-1)$, (ii) and (iv) are equivalent.
\end{proof}

We finish the section with observations on $A^p_\om$ with $1<p<\infty$ and $\om\in\DD$. The method of proof we employed for $D^p_\om$ certainly works also for $A^p_\om$. The analogue of Theorem~\ref{TgDpom} states that
	\begin{equation*}
	\|T_g\|_{A^p_{\om}\rightarrow H^{\infty}}^{p'}
	\asymp\sum_{k=0}^{\infty}\frac{(k+1)^{p'-2}}{(\omega_k)^{p'-1}}
	\left(\sum_{n=0}^{\infty}\frac{(n+1)\widehat{g}(n+1)}{n+k+1}\right)^{p'},
	\end{equation*}
while the argument used to obtain Theorem~\ref{Dpconstant} shows that $T(A^p_{\om}, H^{\infty})$ (equivalently $T_c(A^p_{\om}, H^{\infty})$) consists of constant functions only if and only if
	\begin{equation}\label{fiwehlwfiehg}
	\int_0^1\frac{dr}{\widehat{\om}(r)^{p'-1}}=\infty.
	\end{equation}
This last result extends \cite[Theorem~1.3]{CPPR} from the setting of regular weight to the whole doubling class $\DD$. In the case $1<p\le 2$ we can actually easily do better. Namely, if $1<p<\infty$, $\omega\in\DD$ and $g\in H^\infty$, then Theorem~\ref{Dpconstant} implies that $T(D^p_{\widetilde{\om}_{[p]}},H^\infty)$ consists of constant functions only if and only if
	$$
	\int_0^1\frac{(1-r)^{p'}}{\left(\int_r^1\widetilde{\om}_{[p]}(t)\,dt\right)^{p'-1}}\,dr=\infty.
	$$
Since Lemma~\ref{Eqdhat}(ii) yields $\int_r^1\widehat{\om}(t)(1-t)^{p-1}\,dt\asymp\widehat{\om}(r)(1-r)^p$ for all $0\le r<1$, it follows that this condition is equivalent to \eqref{fiwehlwfiehg}. Further, since $D^p_{\widetilde{\om}_{[p]}}$ might be strictly smaller than $A^p_\om$ when $\om\in\DD\setminus\DDD$ by Propositions~\ref{inclusion} and~\ref{H-A-different}, this is indeed an improvement as claimed above. To judge whether or not we may replace $A^p_\om$ by $\HL^{\om}_{p}$ in the case $2<p<\infty$ is left as a task for an interested reader.

\end{document}